\documentclass[a4paper,12pt]{article}
\usepackage{latexsym}
\usepackage{amsmath}
\usepackage{amssymb}
\usepackage{enumerate}
\usepackage{bm}
\usepackage{mathrsfs}

\usepackage{theorem}
\newtheorem{theorem}{Theorem}[section]
\newtheorem{proposition}[theorem]{Proposition}
\newtheorem{lemma}[theorem]{Lemma}
\newtheorem{corollary}[theorem]{Corollary}

\theorembodyfont{\rmfamily}
\newtheorem{proof}{\textmd{\textit{Proof.}}}

\newtheorem{remark}[theorem]{Remark}
\newtheorem{example}[theorem]{Example}
\newtheorem{definition}[theorem]{Definition}
\newtheorem{step}{Step}

\makeatletter

\@addtoreset{equation}{section}
\makeatother

\newcommand{\qedd}{\hfill \Box}
\newcommand{\ve}{\varepsilon}
\newcommand{\del}{\partial}
\newcommand{\lra}{\longrightarrow}
\newcommand{\e}{\mathrm{e}}

\newcommand{\N}{\ensuremath{\mathbb{N}}}
\newcommand{\Z}{\ensuremath{\mathbb{Z}}}
\newcommand{\R}{\ensuremath{\mathbb{R}}}
\newcommand{\Sph}{\ensuremath{\mathbb{S}}}
\newcommand{\cC}{\ensuremath{\mathcal{C}}}
\newcommand{\cG}{\ensuremath{\mathcal{G}}}
\newcommand{\cI}{\ensuremath{\mathcal{I}}}
\newcommand{\cL}{\ensuremath{\mathcal{L}}}
\newcommand{\cP}{\ensuremath{\mathcal{P}}}
\newcommand{\fm}{\ensuremath{\mathfrak{m}}}
\newcommand{\fv}{\ensuremath{\mathfrak{v}}}
\newcommand{\fL}{\ensuremath{\mathfrak{L}}}
\newcommand{\bs}{\ensuremath{\mathbf{s}}}
\newcommand{\bJ}{\ensuremath{\mathbf{J}}}

\def\esssup{\mathop{\mathrm{ess\, sup}}}

\def\id{\mathop{\mathrm{id}}\nolimits}
\def\vol{\mathop{\mathrm{vol}}\nolimits}
\def\diam{\mathop{\mathrm{diam}}\nolimits}
\def\supp{\mathop{\mathrm{supp}}\nolimits}

\def\Ric{\mathop{\mathrm{Ric}}\nolimits}
\def\CD{\mathop{\mathrm{CD}}\nolimits}
\def\CDD{\mathop{\mathrm{CDD}}\nolimits}
\def\RCD{\mathop{\mathrm{RCD}}\nolimits}
\def\Ent{\mathop{\mathrm{Ent}}\nolimits}

\def\Image{\mathop{\mathsf{Image}}\nolimits}
\def\Lip{\mathop{\mathsf{Lip}}\nolimits}
\def\Dom{\mathop{\mathsf{Dom}}\nolimits}
\newcommand{\rev}[1]{\overleftarrow{#1}}

\title{Needle decompositions and isoperimetric inequalities in Finsler geometry}
\author{Shin-ichi Ohta\thanks{Department of Mathematics, Kyoto University,
Kyoto 606-8502, Japan ({\sf sohta@math.kyoto-u.ac.jp});
Supported in part by JSPS Grant-in-Aid for Scientific Research (KAKENHI) 15K04844.}}
\date{}
\begin{document}

\maketitle

\vspace{-7mm}
\begin{center}
{\it  Dedicated to Professor Takao Yamaguchi
on the occasion of his 60th birthday}
\end{center}
\smallskip

\begin{abstract}
Klartag recently gave a beautiful alternative proof of the isoperimetric inequalities
of L\'evy--Gromov, Bakry--Ledoux, Bayle and E.~Milman on weighted Riemannian manifolds.
Klartag's approach is based on a generalization of the localization method
(so-called needle decompositions) in convex geometry,
inspired also by optimal transport theory.
Cavalletti and Mondino subsequently generalized the localization method,
in a different way more directly along optimal transport theory,
to essentially non-branching metric measure spaces satisfying the curvature-dimension condition.
This class in particular includes reversible (absolutely homogeneous) Finsler manifolds.
In this paper, we construct needle decompositions of non-reversible
(only positively homogeneous) Finsler manifolds,
and show an isoperimetric inequality under bounded reversibility constants.
A discussion on the curvature-dimension condition $\CD(K,N)$
for $N=0$ is also included, it would be of independent interest.
\medskip

\noindent
{\it Mathematics Subject Classication (2010)}: 53C60, 49Q20
\end{abstract}

\renewcommand{\contentsname}{{\large Contents}}
{\small
\tableofcontents
}

\section{Introduction}%%%%%%%%%%%%%%%%%%%%%%
%%%%%%%%%%%%%%%%

In a recent paper \cite{Kl}, Klartag gave a beautiful alternative proof of the isoperimetric inequalities
of L\'evy--Gromov (\cite{Lev1,Lev2}, \cite[Appendix~C]{Gr}),
Bakry--Ledoux~\cite{BL}, Bayle~\cite{Bay} and E.~Milman~\cite{Misharp,Mineg}
on weighted Riemannian manifolds with lower Ricci curvature bounds.
His idea, a generalization of the deep \emph{localization method}
\`a la Payne--Weinberger~\cite{PW}, Gromov--Milman~\cite{GM},
Kannan--Lov\'asz--Simonovits~\cite{LS,KLS} in convex geometry,
is reducing the inequality to those on geodesics (called \emph{needles})
forming a geodesic foliation of the space.
Then we apply the one-dimensional isoperimetric inequality intensively studied
in \cite{Misharp,Mineg}.
A crucial point of Klartag's argument is that it does not depend on
the deep regularity theory of isoperimetric minimizers in geometric measure theory
(due to Almgren, Federer, Morgan et al, see \cite{Mo}),
that had played an irreplaceable role in the study of isoperimetric inequalities
under lower Ricci curvature bounds.
This technique also provides (geometric) functional inequalities
such as the Brunn--Minkowski inequality.

Let us briefly explain how to construct a needle decomposition
associated with a $1$-Lipschitz function $\varphi$ on a Riemannian manifold $M$
in the manner we will adopt in \S \ref{sc:needle}.
We call a unit speed geodesic $\gamma:I \lra M$, on a closed interval $I \subset \R$,
a \emph{transport ray} if
\[ \varphi\big( \gamma(t) \big) -\varphi \big( \gamma(s) \big)=t-s \qquad
 \text{for all}\ s,t \in I\ \text{with}\ s<t. \]
Decompose $M$ as $M=\bm{D}_{\varphi} \sqcup \bm{T}_{\varphi} \sqcup \bm{B}_{\varphi}$,
where for $x \in \bm{D}_{\varphi}$ there is no non-constant transport ray including $x$,
and for $x \in \bm{T}_{\varphi}$ (resp.\ $x \in \bm{B}_{\varphi}$)
we have exactly one (resp.\ more than two) transport ray passing through $x$.
The \emph{transport set} $\bm{T}_{\varphi}$ is of our main interest.
The set $\bm{R}_{\varphi}$ of non-constant transport rays
can be regarded as a quotient space $\bm{R}_{\varphi}=\bm{T}_{\varphi}/\sim$.
Equip $M$ with a weighted measure $\fm$ satisfying $\Ric_N \ge K$
($K \in \R$ and $N \in (-\infty,0] \cup [\dim M,\infty]$,
see Definition~\ref{df:wRic} for the definition of $\Ric_N$).
Pushing $\fm|_{\bm{T}_{\varphi}}$ forward to the measure $\fv$ on $\bm{R}_{\varphi}$,
we have a disintegration $\fm|_{\bm{T}_{\varphi}}=\bm{\mu}_{\gamma} \,\fv(d\gamma)$.
Here $\bm{\mu}_{\gamma}$ is regarded as a measure on the domain of
the geodesic $\gamma \in \bm{R}_{\varphi}$,
and enjoys the same curvature bound $\Ric_N \ge K$.
Now the analysis of $(M,\fm)$ is reduced to
the one-dimensional analysis of $\bm{\mu}_{\gamma}$
via the integration with respect to $\fv$.

As stressed in \cite{Kl}, the above construction is closely related to
optimal transport theory for the $L^1$-cost function $c(x,y):=d(x,y)$.
We especially refer to \cite{BC,Ca,Ca2} for studies in metric measure spaces.
In fact, the construction in \cite{BC,Ca,Ca2} applies to less smooth spaces than Klartag's approach
(for instance, it seems difficult to extend
Whitney's extension theorem and the $\cC^{1,1}$-calculus in \cite{Kl}
to non-smooth metric measure spaces).
Developing in this way, Cavalletti and Mondino~\cite{CM} generalized
the localization method to essentially non-branching metric measure spaces 
satisfying Lott--Sturm--Villani's \emph{curvature-dimension condition} $\CD(K,N)$
for $K \in \R$ and $N \in (1,\infty)$
(precisely, the slightly weaker
\emph{reduced curvature-dimension condition} $\CD^*(K,N)$ is enough).
The class of essentially non-branching $\CD^*(K,N)$-spaces includes
limits of (weighted) Riemannian manifolds, finite-dimensional Alexandrov spaces,
metric measure spaces satisfying the \emph{Riemannian curvature-dimension condition}
(see \cite{AGSrcd,AGMR,EKS}),
and reversible Finsler manifolds (with appropriate lower curvature bounds).
For all of these spaces, the isoperimetric inequality obtained in \cite{CM} as an application
had been previously unknown.
Some functional inequalities are also studied in the subsequent paper \cite{CM2}.

The aim of this article is to further extend the localization method to
\emph{non-reversible} Finsler manifolds.
A Finsler manifold $(M,F)$ is a couple of a manifold $M$
and a non-negative $\cC^{\infty}$-function $F:TM \lra [0,\infty)$
giving a Minkowski norm on each tangent space $T_xM$
(see \S \ref{ssc:Fgeom} for the precise definition).
We say that $F$ (or $(M,F)$) is \emph{reversible} if $F(-v)=F(v)$ for all $v \in TM$.
The reversibility is equivalent to the symmetry of the associated distance function:
$d(y,x)=d(x,y)$ for all $x,y \in M$.
In many situations, non-reversible Finsler manifolds behave
equally well as reversible ones.
For instance, when we equip $(M,F)$ with a positive $\cC^{\infty}$-measure $\fm$ on $M$,
the weighted Ricci curvature bound $\Ric_N \ge K$ is equivalent to
the curvature-dimension condition $\CD(K,N)$
(see \cite{Oint} for $N \in [\dim M,\infty]$, \cite{Oneg} for $N \in (-\infty,0)$,
and \S \ref{sc:CD} in this paper for $N=0$).

For needle decompositions,
although it is unclear whether Klartag's original construction is extended to Finsler manifolds,
the more abstract way due to Cavalletti et al~\cite{BC,Ca,CM} is available.

\begin{theorem}[Needle decompositions]\label{th:needle}
Let $(M,F)$ be a connected, forward and backward complete,
$n$-dimensional $\cC^{\infty}$-Finsler manifold
of $\del M=\emptyset$ and $n \ge 1$,
endowed with a positive $\cC^{\infty}$-measure $\fm$ on $M$.
\begin{enumerate}[{\rm (i)}]
\item
Given a $1$-Lipschitz function $\varphi:M \lra \R$, we have a decomposition
$M=\bm{D}_{\varphi} \sqcup \bm{T}_{\varphi} \sqcup \bm{B}_{\varphi}$
such that $\fm(\bm{B}_{\varphi})=0$ and that the set $\bm{R}_{\varphi}$
of non-constant transport rays is identified with a quotient $\bm{T}_{\varphi}/\sim$.
Moreover, there exists a measure $\fv$ on $\bm{R}_{\varphi}$ satisfying
\[ \fm|_{\bm{T}_{\varphi}}=\bm{\mu}_{\gamma} \,\fv(d\gamma), \]
where $\bm{\mu}_{\gamma}$ is a probability measure on $\Image(\gamma) \subset M$
for $\fv$-almost all $\gamma \in \bm{R}_{\varphi}$.

\item
For $\fv$-almost every $\gamma \in \bm{R}_{\varphi}$,
we have $\supp\bm{\mu}_{\gamma}=\Image(\gamma)$ and
$\bm{\mu}_{\gamma}$ is absolutely continuous with respect to
the one-dimensional Lebesgue measure $\fL^1$ on $\Dom(\gamma) \subset \R$
$($identified with $\Image(\gamma))$ with a continuous density function $\rho_{\gamma}$,
namely $\bm{\mu}_{\gamma}=\rho_{\gamma} \cdot \fL^1|_{\Dom(\gamma)}$.

\item
If in addition $(M,F,\fm)$ satisfies $\Ric_N \ge K$ for $K \in \R$ and $N \in (-\infty,0] \cup [n,\infty]$
$($with $N \neq 1$ when $n=1)$,
then for $\fv$-almost every $\gamma \in \bm{R}_{\varphi}$
the density function $\rho_{\gamma}$ satisfies
\[ \rho_{\gamma}\big( (1-\lambda)s+\lambda t \big)
 \ge \left\{ \bm{\sigma}_{K/(N-1)}^{(1-\lambda)}(t-s) \rho_{\gamma}(s)^{1/(N-1)}
 +\bm{\sigma}_{K/(N-1)}^{(\lambda)}(t-s) \rho_{\gamma}(t)^{1/(N-1)} \right\}^{N-1} \]
if $N \in (-\infty,0] \cup [n,\infty);$ and
\[ \log\rho_{\gamma}\big( (1-\lambda)s+\lambda t \big)
 \ge (1-\lambda) \log\rho_{\gamma}(s) +\lambda \log\rho_{\gamma}(t)
 +\frac{K}{2}(1-\lambda)\lambda (t-s)^2 \]
if $N=\infty;$
both for all $a<s<t<b$ with $a,b \in \Dom(\gamma)$ and $\lambda \in (0,1)$.
\end{enumerate}
\end{theorem}

The inequalities in (iii) means that each needle $(\Image(\gamma),F,\bm{\mu}_{\gamma})$
again satisfies $\Ric_N \ge K$.
We remark that the special case of $n=N=1$ is reduced to $\R$ or $\Sph^1$
equipped with the Lebesgue measure (hence $\Ric_1 \equiv 0$)
and easily analyzed (see Remark~\ref{rm:wRic}(c)).

We prove (i) in \S \ref{sc:needle}, (ii) in \S \ref{ssc:qual}, and (iii) in \S \ref{ssc:quan}.
Notice that no information of $\bm{D}_{\varphi}$ (which is in general a large set)
is obtained from Theorem~\ref{th:needle}.
In applications one can ignore $\bm{D}_{\varphi}$ thanks to the following version of
needle decompositions \emph{conditioned by mean-zero functions} (see \S \ref{sc:mean0}).

\begin{theorem}[Needle decompositions conditioned by mean-zero functions]\label{th:mean0}
$\,$\linebreak
Let $(M,F,\fm)$ be as in Theorem~$\ref{th:needle}$ and
$f \in L^1(M;\fm)$ satisfy $\int_M f \,d\fm=0$ and
\[ \int_M |f(x)| \{d(x_0,x)+d(x,x_0)\} \,\fm(dx)<\infty \qquad
 \text{for some}\ x_0 \in M. \]
Take a $1$-Lipschitz function $\varphi:M \lra \R$ maximizing the integral
$\int_M f\varphi \,d\fm$ among all $1$-Lipschitz functions on $M$.
Then the needle decomposition given by Theorem~$\ref{th:needle}$ satisfies
\[ \int_{\Image(\gamma)} f \,d\bm{\mu}_{\gamma}=0 \]
for $\fv$-almost all $\gamma \in \bm{R}_{\varphi}$,
and $f \equiv 0$ $\fm$-almost everywhere on $\bm{D}_{\varphi}$.
\end{theorem}

In particular, if $f$ is never being $0$
(which is the case in our application to isoperimetric inequalities),
then Theorem~\ref{th:mean0} ensures $\fm(\bm{D}_{\varphi})=0$.

In order to state the isoperimetric inequalities, we need some notations.
Suppose $\fm(M)<\infty$ and normalize $\fm$ as $\fm(M)=1$
(such a normalization does not change $\Ric_N$).
For a Borel set $A \subset M$,
define an analogue of the \emph{Minkowski exterior boundary measure} as
\[ \fm^+(A):=\liminf_{\ve \downarrow 0} \frac{\fm(B^+(A,\ve))-\fm(A)}{\ve}, \]
where $B^+(A,\ve):=\{ y \in M \,|\, \inf_{x \in A}d(x,y)<\ve \}$
is the forward $\ve$-neighborhood of $A$.
Then the \emph{isoperimetric profile} $\cI_{(M,F,\fm)}:[0,1] \lra [0,\infty]$
of $(M,F,\fm)$ is defined by, for $\theta \in [0,1]$,
\begin{equation}\label{eq:I}
\cI_{(M,F,\fm)}(\theta):=\inf\{\fm^+(A) \,|\, A \subset M \text{: Borel sets with}\ \fm(A)=\theta\}.
\end{equation}
We have $\cI_{(M,F,\fm)}(0)=\cI_{(M,F,\fm)}(1)=0$
by taking $A=\emptyset$ and $A=M$, respectively.
The isoperimetric profile of (weighted) Riemannian manifolds is a classical research object,
and was intensively studied by E.~Milman~\cite{Misharp,Mineg}
under the combination of $\Ric_N \ge K$ and $\diam M :=\sup_{x,y \in M}d(x,y) \le D$
for $K \in \R$, $N \in (-\infty,1) \cup [n,\infty]$, $D \in (0,\infty]$
(so-called the \emph{curvature-dimension-diameter condition} $\CDD(K,N,D)$).
He showed that weighted Riemannian manifolds with $\Ric_N \ge K$ and $\diam M \le D$
enjoy the same isoperimetric inequality
\begin{equation}\label{eq:M-isop}
\cI_{(M,g,\fm)}(\theta) \ge \cI_{K,N,D}(\theta) \qquad \text{for all}\ \theta \in [0,1]
\end{equation}
regardless the dimension $n$ of the spaces
(we remark that, to be precise, the case of $N \in (0,1)$ and $D<\infty$
was excluded in \cite{Mineg}, see \cite[Remark~1.5]{Mineg}).
Furthermore, \eqref{eq:M-isop} is sharp in all parameters $K$, $N$ and $D$ in all dimensions $n$.
In fact, in \cite{Misharp,Mineg} the precise formula of $\cI_{K,N,D}$
in terms of the isoperimetric profile $\cI^{\flat}$ on intervals $I$
tested for $A \subset I$ such that $\del A$ is a point
(that is, $A=[0,a)$ or $(a,D]$ if $I=[0,D]$)
and model spaces assuring the sharpness are given.
See \cite{Misharp,Mineg} for the precise formula of $\cI_{K,N,D}$,
here we mention only the two classical cases:
\[ \cI_{K,N,D}(\theta)
 =\frac{\sin(\sqrt{K/(N-1)}R(\theta))^{N-1}}{\int_0^{\pi\sqrt{(N-1)/K}} \sin(\sqrt{K/(N-1)}r)^{N-1} \,dr} \]
for $K>0$, $N \in [n,\infty)$ and $D \ge \pi \sqrt{(N-1)/K}$,
where $R(\theta) \in [0,\pi\sqrt{(N-1)/K}]$ is given by
\[ \theta=\frac{\int_0^{R(\theta)} \sin(\sqrt{K/(N-1)}r)^{N-1} \,dr}
 {\int_0^{\pi\sqrt{(N-1)/K}} \sin(\sqrt{K/(N-1)}r)^{N-1} \,dr}; \]
and
\[ \cI_{K,\infty,\infty}(\theta) =\sqrt{\frac{K}{2\pi}} \e^{-Ka(\theta)^2/2}, \qquad
 \text{where}\ \theta=\int_{-\infty}^{a(\theta)} \sqrt{\frac{K}{2\pi}} \e^{-Ks^2/2} \,ds, \]
for $K>0$.
The first case corresponds to \emph{L\'evy--Gromov's isoperimetric inequality}
(extended by Bayle~\cite{Bay} to the weighted situation and non-integer $N$)
employing the spheres of constant curvature as model spaces.
The second case is \emph{Bakry--Ledoux's isoperimetric inequality},
where model spaces are Euclidean spaces with Gaussian distributions.
As an application of Theorems~\ref{th:needle}, \ref{th:mean0},
we give a Finsler version of \eqref{eq:M-isop} as follows.

\begin{theorem}[Isoperimetric inequality]\label{th:isop}
Let $(M,F,\fm)$ be as in Theorem~$\ref{th:needle}$ with $\fm(M)=1$ and
assume $\Ric_N \ge K$ and $\diam M \le D$
for some $K \in \R$, $N \in (-\infty,0] \cup [n,\infty]$ $($with $N \neq 1$ if $n=1)$
and $D \in (0,\infty]$ as well as
\[ \Lambda_{(M,F)} :=\sup_{v \in TM \setminus 0} \frac{F(-v)}{F(v)} <\infty. \]
Then we have
\begin{equation}\label{eq:Fisop}
\cI_{(M,F,\fm)}(\theta) \ge \Lambda_{(M,F)}^{-1} \cdot \cI_{K,N,D}(\theta)
\end{equation}
for all $\theta \in [0,1]$.
If $n=N=1$, then necessarily $K=0$, $D<\infty$ and we have for all $\theta \in (0,1)$
\[ \cI_{(M,F,\fm)}(\theta) =\frac{1+\Lambda_{(M,F)}}{D}. \]
\end{theorem}

We call $\Lambda_{(M,F)}$ the \emph{reversibility constant}.
Clearly $\Lambda_{(M,F)} \ge 1$, and $\Lambda_{(M,F)}=1$ holds if and only if $F$ is reversible.
In the reversible case included in \cite{CM},
we have the sharp inequality $\cI_{(M,F,\fm)}(\theta) \ge \cI_{K,N,D}(\theta)$
same as the Riemannian case.
In the non-reversible case, however, needle decompositions seem to have
a limited strength and gives only the weaker estimate \eqref{eq:Fisop}.
This is because the reverse curve $\bar{\gamma}(t):=\gamma(l-t)$
of a geodesic $\gamma:[0,l] \lra M$ is not necessarily geodesic in the non-reversible situation.
Hence in Theorem~\ref{th:needle}(iii)
we have no information of $\rho_{\gamma}$ along the reverse curve of $\gamma$
(parametrized by arc-length).

Though our construction of needle decompositions essentially follows the lines of \cite{BC,Ca,CM},
one can give simpler and clearer descriptions at some points
thanks to finer properties of Finsler manifolds
such as the better understanding of the behavior of geodesics.
For the sake of accessibility (to Finsler geometers for instance),
we tried to make this article self-contained
up to some basic facts about the weighted Ricci curvature and optimal transport theory
(these can be found in \cite{Oint}, see also \cite{Oneg} for the case of $N<0$).
We also believe that this alternative approach to Klartag's work \cite{Kl}
is worthwhile even in the Riemannian case.

The article is organized as follows.
After preliminaries on Finsler geometry and optimal transport theory,
in \S \ref{sc:CD} we discuss Lott--Sturm--Villani's curvature-dimension condition $\CD(K,N)$
for $K \in \R$, $N \in (-\infty,0] \cup [n,\infty]$.
The case of $N=0$ is new and of independent interest.
We construct a needle decomposition associated with a $1$-Lipschitz function
and prove Theorem~\ref{th:needle}(i) in \S \ref{sc:needle}.
\S \ref{sc:mean0} is devoted to the proof of Theorem~\ref{th:mean0}.
In \S \ref{sc:prop} we come back to the study of general $1$-Lipschitz functions
and show Theorem~\ref{th:needle}(ii), (iii).
We prove Theorem~\ref{th:isop} in \S \ref{sc:isop},
and close the article with several further problems in \S \ref{sc:prob}.
\medskip

{\it Acknowledgements.}
I am grateful to Emanuel Milman for suggesting to learn Klartag's work \cite{Kl} to me,
and to Andrea Mondino for fruitful discussions on \cite{CM}.
I would like to also thank my colleagues in Kyoto University,
Takao Yamaguchi, Takumi Yokota and Yu Kitabeppu,
for sharing the interest in this work.

\section{Preliminaries}\label{sc:prel}%%%%%%%%%%%%%%%%%%%%%%
%%%%%%%%%%%%%%%%

We review basic facts in Finsler geometry and optimal transport theory
necessary in our discussion.

\subsection{Finsler geometry}\label{ssc:Fgeom}%%%%%%%%%%%%%%%%%%%%%%
%%%%%%%%%%%%%%%%

We refer to \cite{BCS,Shlec,Oint} for materials in this subsection.
Let $M$ be a connected $\cC^{\infty}$-manifold of dimension $n \ge 1$ without boundary.
Given a local coordinate $(x^i)_{i=1}^n$ on an open set $U \subset M$,
we will always use the fiber-wise linear coordinate $(x^i,v^j)_{i,j=1}^n$ of $TU$ such that
\[ v=\sum_{j=1}^n v^j \frac{\del}{\del x^j}\Big|_x \in T_xM
 \qquad \text{for}\ x \in U. \]

\subsubsection{Finsler manifolds}%%%%%%%%%%%%%%%%

\begin{definition}[Finsler structures]\label{df:Fstr}
A nonnegative function $F:TM \lra [0,\infty)$ is called
a \emph{$\cC^{\infty}$-Finsler structure} of $M$ if the following three conditions hold.
\begin{enumerate}[(1)]
\item(\emph{Regularity})
$F$ is $\cC^{\infty}$ on $TM \setminus 0$,
where $0$ stands for the zero section.

\item(\emph{Positive $1$-homogeneity})
$F(cv)=cF(v)$ holds for all $v \in TM$ and $c>0$.

\item(\emph{Strong convexity})
The $n \times n$ matrix
\begin{equation}\label{eq:gij}
\big( g_{ij}(v) \big)_{i,j=1}^n :=
 \bigg( \frac{1}{2}\frac{\del^2 (F^2)}{\del v^i \del v^j}(v) \bigg)_{i,j=1}^n
\end{equation}
is positive-definite for all $v \in TM \setminus 0$.
\end{enumerate}
We call such a pair $(M,F)$ a \emph{$\cC^{\infty}$-Finsler manifold}.
\end{definition}

\begin{remark}\label{rm:rev}
We stress that the homogeneity is required only in the positive direction,
therefore $F(-v) \neq F(v)$ is allowed.
Admitting such \emph{non-reversibility} is one of the important features of Finsler manifolds
(see \cite[Chapter~11]{BCS} for an important class of non-reversible Finsler manifolds
called \emph{Randers spaces}).
We say that $F$ is \emph{reversible} if $F(-v)=F(v)$ holds for all $v \in TM$
(in other words, $F$ is \emph{absolutely $1$-homogeneous}).
\end{remark}

For each $v \in T_xM \setminus 0$, the positive-definite matrix
$(g_{ij}(v))_{i,j=1}^n$ in \eqref{eq:gij} induces
the Riemannian structure $g_v$ of $T_xM$ by
\begin{equation}\label{eq:gv}
g_v\bigg( \sum_{i=1}^n a_i \frac{\del}{\del x^i}\Big|_x,
 \sum_{j=1}^n b_j \frac{\del}{\del x^j}\Big|_x \bigg)
 := \sum_{i,j=1}^n a_i b_j g_{ij}(v).
\end{equation}
This inner product is regarded as the best Riemannian approximation of $F|_{T_xM}$ in the direction $v$,
and plays a vital role in Finsler geometry.
A  geometric way of introducing $g_v$ is that the unit sphere of $g_v$ is tangent to that of $F|_{T_xM}$
at $v/F(v)$ up to the second order.
In particular, we have $g_v(v,v)=F(v)^2$.

For $x,y \in M$, define the \emph{distance} from $x$ to $y$ in a natural way by
\[ d(x,y):=\inf_{\eta} \int_0^1 F\big( \dot{\eta}(t) \big) \,dt, \]
where the infimum is taken over all $\cC^1$-curves $\eta:[0,1] \lra M$
with $\eta(0)=x$ and $\eta(1)=y$.
We remark that our distance can be \emph{asymmetric} (namely $d(y,x) \neq d(x,y)$)
since $F$ is only positively homogeneous.
The following (ordered) triangle inequality is readily observed from the definition:
\begin{equation}\label{eq:tri}
d(x,y) \le d(x,z)+d(z,y) \qquad \text{for all}\ x,y,z \in M.
\end{equation}

A $\cC^{\infty}$-curve $\eta:[0,l] \lra M$ is called a \emph{geodesic}
if it is locally minimizing and has a constant speed
(meaning that $F(\dot{\eta})$ is constant).
We remark that the reverse curve $\bar{\eta}(t):=\eta(l-t)$
is not necessarily locally minimizing nor of constant speed because of the non-reversibility of $F$.
One can write down the geodesic equation,
then the standard ODE theory ensures the short time existence and the uniqueness
of a geodesic for each given initial velocity.
Given $v \in T_xM$, if there is a geodesic $\eta:[0,1] \lra M$
with $\dot{\eta}(0)=v$, then we define the \emph{exponential map}
by $\exp_x(v):=\eta(1)$.
We say that $(M,F)$ is \emph{forward complete} if the exponential
map is defined on whole $TM$.
Then by the Hopf--Rinow theorem any pair of points
is connected by a minimal geodesic (see \cite[Theorem~6.6.1]{BCS}).

\subsubsection{Lipschitz functions}%%%%%%%%%%%%%%%%%%%%%%

Let us denote by $\cL^*:T^*M \lra TM$ the \emph{Legendre transform}
associated with $F$ and its dual norm $F^*$ on $T^*M$.
Precisely, $\cL^*$ is sending $\alpha \in T_x^*M$ to the unique element $v \in T_xM$
such that $\alpha(v)=F^*(\alpha)^2$ and $F(v)=F^*(\alpha)$.
Note that $\cL^*|_{T^*_xM}$ becomes a linear operator only when $F|_{T_xM}$
is an inner product.
For a differentiable function $\varphi:M \lra \R$, the \emph{gradient vector}
of $\varphi$ at $x$ is defined as the Legendre transform of the derivative:
$\nabla \varphi(x):=\cL^*(d\varphi(x)) \in T_xM$.

With respect to our asymmetric distance $d$,
we say that a function $\varphi:M \lra \R$ is \emph{$L$-Lipschitz} for $L \ge 0$ if
\begin{equation}\label{eq:L-Lip}
-Ld(y,x) \le \varphi(y)-\varphi(x) \le Ld(x,y) \qquad \text{for all}\ x,y \in M.
\end{equation}
Notice that the first inequality in \eqref{eq:L-Lip} indeed follows from the second one
by exchanging $x$ and $y$.
If $\varphi$ is $\cC^1$, then \eqref{eq:L-Lip} is equivalent to $\sup_M F(\nabla \varphi) \le L$.

We denote by $\Lip_L(M)$ the set of all $L$-Lipschitz functions on $M$.
We will be mainly interested in the case of $L=1$.
A typical example of a $1$-Lipschitz function is the distance function from a set:
$\varphi_A(x)=\inf_{z \in A}d(z,x)$ with $A \subset M$.
The triangle inequality \eqref{eq:tri} ensures that $\varphi_A \in \Lip_1(M)$.

\subsubsection{Weighted Ricci curvature}%%%%%%%%%%%%%%%%%%%%%

Next we discuss the curvature.
When $n=1$, we consider the Ricci curvature to be identically zero.
When $n \ge 2$,
the Ricci curvature for a Finsler manifold is defined by using the Chern connection.
Instead of giving the precise definition,
here we explain a useful interpretation found in \cite[\S 6.2]{Shlec}
(going back to \cite{Au}).

We denote the \emph{unit tangent sphere bundle} by $UM:=TM \cap F^{-1}(1)$.
Given $v \in U_xM$, we extend it to a $\cC^{\infty}$-vector field $V$
on a neighborhood of $x$ in such a way that every integral curve of $V$ is geodesic,
and consider the Riemannian structure $g_V$ induced from \eqref{eq:gv}.
Then the \emph{Ricci curvature} $\Ric(v)$ of $v$ ($=V(x)$) with respect to $F$ coincides with
the Ricci curvature of $v$ with respect to $g_V$
(independently from the choice of $V$).

Now we fix a positive $\cC^{\infty}$-measure $\fm$ on $M$.
Inspired by the above interpretation of the Finsler-Ricci curvature
and the theory of weighted Riemannian manifolds,
the weighted Ricci curvature for the triple $(M,F,\fm)$ was introduced in \cite{Oint} as follows.

\begin{definition}[Weighted Ricci curvature]\label{df:wRic}
We first define the function $\Psi:UM \lra \R$
by the decomposition $\fm=\e^{-\Psi(\dot{\eta})}\vol_{\dot{\eta}}$
along unit speed geodesics $\eta$,
where $\vol_{\dot{\eta}}$ denotes the Riemannian volume measure of $g_{\dot{\eta}}$.
Then, given a unit vector $v \in U_xM$ and the geodesic $\eta:(-\ve,\ve) \lra M$
with $\dot{\eta}(0)=v$, we define the \emph{weighted Ricci curvature}
involving a parameter $N \in (-\infty,0] \cup [n,\infty]$ by
\begin{enumerate}[(1)]
\item $\Ric_N(v):=\Ric(v) +(\Psi \circ \dot{\eta})''(0) -\displaystyle\frac{(\Psi \circ \dot{\eta})'(0)^2}{N-n}\quad$
for $N \in (-\infty,0] \cup (n,\infty)$,

\item $\Ric_{\infty}(v):=\Ric(v) +(\Psi \circ \dot{\eta})''(0)$,

\item $\Ric_n(v):=\displaystyle
 \begin{cases} \Ric(v)+(\Psi \circ \dot{\eta})''(0) &\ \text{if}\ (\Psi \circ \dot{\eta})'(0)=0, \\
 -\infty &\ \text{if}\ (\Psi \circ \dot{\eta})'(0) \neq 0. \end{cases}$
\end{enumerate}
We also set $\Ric_N(cv):=c^2 \Ric_N(v)$ for $c \ge 0$.
\end{definition}

We will say that $\Ric_N \ge K$ holds for some $K \in \R$
if $\Ric_N(v) \ge KF^2(v)$ for all $v \in TM$.

\begin{remark}\label{rm:wRic}
(a)
In the notation of Definition~\ref{df:wRic},
$(\Psi \circ \dot{\eta})'(0)$ coincides with the \emph{$\mathbf{S}$-curvature}
$\mathbf{S}(v)$ (see \cite[\S 7.3]{Shlec}).
For a Riemannian manifold $(M,g,\vol_g)$ endowed with the Riemannian volume measure,
clearly we have $\Psi \equiv 0$ and hence $\Ric_N =\Ric$ for all $N$.
We know that, however, a Finsler manifold may not admit any measure $\fm$ satisfying $\mathbf{S} \equiv 0$
(in other words, $\Ric_n \neq -\infty$), see \cite{ORand} for such an example.
This means that there may be no nice reference measure,
thus it is natural (and necessary) to begin with an arbitrary measure.

(b)
Although we will consider only $N \in (-\infty,0] \cup [n,\infty]$,
the definition of $\Ric_N$ in (1) makes sense also for $N \in (0,n)$.
We observe from the definition that
\[ \Ric_n \le \Ric_N \le \Ric_{\infty} \le \Ric_{N'} \qquad
 \text{for}\ n<N<\infty,\, -\infty<N'<n, \]
and $\Ric_N, \Ric_{N'}$ are non-decreasing in $N \in [n,\infty]$, $N' \in (-\infty,n)$, respectively.
Traditionally the range of $N$ was restricted in $[n,\infty]$ (see \cite{Bak,Qi,Lo}).
The case of $N \in (-\infty,0)$ was investigated rather recently in
\cite{OT1,OT2,MR,KM,Oneg,Mineg}, and some results admit $N \in [0,1)$
and even $N=1$ (see \cite{Wy3}).
Klartag's work \cite{Kl} also covers $N \in (-\infty,1) \cup [n,\infty]$.
See \cite{Miex} for a recent interesting example equipped with $N \in (-\infty,n)$.

(c)
If $n=N=1$, then $\Ric_1 \ge K$ implies $(\Psi \circ \dot{\eta})' \equiv 0$
and hence $\Psi$ is constant on each of the two connected components of $UM$.
This means that
\[ F\bigg( \frac{\del}{\del x} \bigg) \equiv c_1, \qquad
 F\bigg( {-}\frac{\del}{\del x} \bigg) \equiv c_2, \qquad
 d\fm=c_3 dx \]
for some $c_1,c_2,c_3>0$, where $x$ is the standard coordinate of $M=\R$ or $\Sph^1$.
In particular, $\Ric_1 \equiv 0$ in this case.
\end{remark}

Similarly to the weighted Riemannian case,
the bound $\Ric_N \ge K$ implies many analytic and geometric consequences,
such as Bochner's inequality (for $N \in (-\infty,0) \cup [n,\infty]$, see \cite{OSbw,Oneg}),
the Bishop--Gromov volume comparison theorem (for $N \in [n,\infty)$, see \cite{Oint}),
and the Cheeger--Gromoll splitting theorem (for $N \in [n,\infty]$, see \cite{Osplit}).
Bochner's inequality and the splitting theorem can be generalized even for $N \in (-\infty,1]$,
see \cite{Wy3}.

For later convenience, we introduce the following notations.

\begin{definition}[Reverse Finsler structures]\label{df:rev}
Define the \emph{reverse Finsler structure} $\rev{F}$ of $F$ by $\rev{F}(v):=F(-v)$.
We will put arrows $\leftarrow$ on those quantities associated with $\rev{F}$,
for example, $\rev{d}\!(x,y)=d(y,x)$, $\rev{\nabla}\varphi=-\nabla(-\varphi)$
and $\rev{\Ric}_N(v)=\Ric_N(-v)$.
\end{definition}

Observe that $\varphi$ is $1$-Lipschitz with respect to $F$ if and only if
$-\varphi$ is $1$-Lipschitz with respect to $\rev{F}$.
Notice also that $\Ric_N \ge K$ is equivalent to $\rev{\Ric}_N(v) \ge K\rev{F}(v)^2$,
and hence the weighted Ricci curvature bound is common between $F$ and $\rev{F}$.
We say that $(M,F)$ is \emph{backward complete} if $(M,\rev{F})$ is forward complete.
We remark that the forward and backward completenesses are not mutually equivalent in general.

\subsection{Optimal transport theory}\label{ssc:OT}%%%%%%%%%%%%%%%%%%%%
%%%%%%%%%%%%%%%%

We refer to \cite{Vi1,Vi2} for the basics and recent developments of optimal transport theory.
Here we restrict ourselves to the case of a Finsler manifold $(M,F,\fm)$ as in the previous subsection.

Denote by $\cP(M)$ the set of all Borel probability measures on $M$.
For $p \in [1,\infty)$, let
$\cP^p(M) \subset \cP(M)$ be the subset consisting of measures $\mu$ satisfying
\[ \int_M \{d^p(x,y) +d^p(y,x)\} \,\mu(dy) <\infty \]
for some (and hence all) $x \in M$.
For $\mu,\nu \in \cP(M)$, we say that $\pi \in \cP(M \times M)$ is their \emph{coupling} if
$(p_1)_{\sharp}\pi=\mu$ and $(p_2)_{\sharp}\pi=\nu$,
where $p_1,p_2:M \times M \lra M$ are the projections ($p_1(x,y)=x,p_2(x,y)=y$)
and $(p_i)_{\sharp}\pi$ denotes the push-forward measure of $\pi$ by $p_i$.

The \emph{$L^p$-Wasserstein distance} between $\mu,\nu \in \cP^p(M)$ is defined by
\[ W_p(\mu,\nu):=\inf_{\pi \in \Pi(\mu,\nu)}
 \left( \int_{M \times M} d^p(x,y) \,\pi(dxdy) \right)^{1/p}, \]
where $\Pi(\mu,\nu) \subset \cP(M \times M)$ is the set of all couplings of $\mu$ and $\nu$.
A coupling attaining the above infimum is called a \emph{$d^p$-optimal coupling}.
Notice that $W_p(\mu,\nu)<\infty$ by the definition of $\cP^p(M)$,
and $W_p$ enjoys the positivity ($W_p(\mu,\nu)>0$ unless $\mu=\nu$ in law)
and the (ordered) triangle inequality
\[ W_p(\mu,\nu) \le W_p(\mu,\omega)+W_p(\omega,\nu) \qquad
 \text{for all}\ \mu,\omega,\nu \in \cP^p(M), \]
while $W_p(\mu,\nu) \neq W_p(\nu,\mu)$ due to the asymmetry of $d$.
Finding an optimal coupling $\pi$ of given $\mu,\nu \in \cP^p(M)$
is called the \emph{Monge--Kantorovich optimal transport problem}.

The following elegant equivalent condition to the optimality of a coupling
will play a role (see \cite[Theorem~5.10]{Vi2}).

\begin{theorem}[Cyclical monotonicity]\label{th:cymo}
For $\mu,\nu \in \cP^p(M)$, a coupling
$\pi \in \Pi(\mu,\nu)$ is $d^p$-optimal if and only if the set $\supp\pi \subset M \times M$
is \emph{$d^p$-cyclically monotone} in the sense that,
for any finite set $\{(x_i,y_i)\}_{i=1}^l \subset \supp\pi$, we have
\[ \sum_{i=1}^l d^p(x_i,y_i) \le \sum_{i=1}^l d^p(x_i,y_{i+1}), \]
where $y_{l+1}:=y_1$ in the RHS.
\end{theorem}

Although our discussion is much indebted to ideas from optimal transport theory
(especially the $p=1$ situation), we will use only a few basic facts of the theory.
Besides Theorem~\ref{th:cymo}, what we need is the fact
(called the \emph{Brenier--McCann theorem} due to \cite{Br,Mc},
see \cite{Oint} for the Finsler case) that any $d^2$-optimal coupling
between $\mu,\nu \in \cP^2(M)$, $\mu$ being absolutely continuous
with respect to $\fm$ (denoted by $\mu \ll \fm$),
is unique and represented by using a measurable map $T:M \lra M$ as
$\pi=(\id_M \times T)_{\sharp}\mu$.
In particular, we have $T_{\sharp}\mu=\nu$.
In this case we call $T$ a \emph{$d^2$-optimal transport} from $\mu$ to $\nu$.
By denoting a minimal geodesic from $x$ to $T(x)$ by $\eta_x:[0,1] \lra M$
(which is unique for $\mu$-almost every $x$)
and putting $T_{\lambda}(x):=\eta_x(\lambda)$,
the curve $\mu_{\lambda}:=(T_{\lambda})_{\sharp}\mu$ in $\cP^2(M)$ clearly satisfies
$\mu_0=\mu,\mu_1=\nu$ and
\[ W_2(\mu_{\lambda},\mu_{\lambda'})=(\lambda' -\lambda)W_2(\mu,\nu) \qquad
 \text{for}\ 0 \le \lambda<\lambda' \le 1. \]
Therefore $(\mu_{\lambda})_{\lambda \in [0,1]}$ is a unique minimal geodesic
from $\mu$ to $\nu$ with respect to $W_2$.
It also holds that $\mu_{\lambda} \ll \fm$ for all $\lambda \in [0,1)$.

Another fact behind our construction is the \emph{Kantorovich--Rubinstein duality}:
\begin{equation}\label{eq:KR}
W_1(\mu,\nu)=\sup_{\phi \in \Lip_1(M)} \left\{ 
 \int_M \phi \,d\nu -\int_M \phi \,d\mu \right\}
 \qquad \text{for}\ \mu,\nu \in \cP^1(M)
\end{equation}
(see \cite[Theorem~5.10]{Vi2},
and also \cite[p.~53]{Vi2} for an interesting economic explanation).
Though we will not use \eqref{eq:KR},
the discussion in \S \ref{sc:mean0} is better understood keeping \eqref{eq:KR} in mind.

\section{Curvature-dimension condition}\label{sc:CD}%%%%%%%%%%%%%%%%%%%%
%%%%%%%%%%%%%%%%

We next discuss the curvature-dimension condition
in the sense of Sturm and Lott--Villani.
This theory is making a breathtaking progress in this decade.
We refer to the book \cite{Vi2} for a detailed account at that time (2009),
and to \cite{Ogren} for a survey from a more geometric viewpoint.

Hereafter,
let $(M,F)$ be a connected, forward and backward complete
$\cC^{\infty}$-Finsler manifold of dimension $n \ge 1$ without boundary,
and let $\fm$ be a positive $\cC^{\infty}$-measure on $M$.
For $\mu \in \cP(M)$ such that $\mu=\rho\fm \ll \fm$,
define the \emph{relative entropy} with respect to $\fm$ by
\[ \Ent_{\fm}(\mu):=\int_M \rho \log\rho \,d\fm \]
if $\int_{\{\rho>1\}} \rho \log\rho \,d\fm<\infty$,
and $\Ent_{\fm}(\mu):=\infty$ otherwise.
We also define for $N \in (-\infty,0) \cup (1,\infty)$ and $\mu=\rho\fm \in \cP(M)$
the (relative) \emph{R\'enyi entropy} with respect to $\fm$ by
\begin{align*}
S_N(\mu) &:=-\int_M \rho^{(N-1)/N} \,d\fm \qquad \text{if}\ N \in (1,\infty), \\
S_N(\mu) &:=\int_M \rho^{(N-1)/N} \,d\fm \qquad \text{if}\ N \in (-\infty,0).
\end{align*}
We suppressed the dependence on $\fm$ for notational simplicity.
Notice that the generating functions $h(s)=s\log s$ ($N=\infty$),
$-s^{(N-1)/N}$ ($N>1$) and
$s^{(N-1)/N}$ ($N<0$) are all convex on $(0,\infty)$ and $h(0)=0$.
In the special case of $N=0$, the entropy $S_0$ is defined as the limit:
\[ S_0(\mu):=\esssup \rho =\lim_{N \uparrow 0} S_N(\mu)^{-N}.  \]

The lower curvature bound $\Ric_N \ge K$ is characterized
by a convexity inequality of $\Ent_{\fm}$ or $S_N$ as follows.
The inequality involves the functions:
\begin{equation}\label{eq:bs}
\bs_{\kappa}(r):= \left\{
 \begin{array}{cl}
 \frac{1}{\sqrt{\kappa}} \sin(\sqrt{\kappa}r) & \text{if}\ \kappa>0, \\
 r & \text{if}\ \kappa=0, \\
 \frac{1}{\sqrt{-\kappa}} \sinh(\sqrt{-\kappa}r) & \text{if}\ \kappa<0,
 \end{array} \right.
 \quad \text{for}\ r \ge 0
\end{equation}
(this is the solution to the Jacobi equation $f''+\kappa f=0$
with $f(0)=0$ and $f'(0)=1$),
\begin{equation}\label{eq:btau}
\bm{\tau}^{(\lambda)}_{K,N}(r)
 :=\lambda^{1/N} \left( \frac{\bs_{K/(N-1)}(\lambda r)}{\bs_{K/(N-1)}(r)} \right)^{(N-1)/N},
 \qquad \lambda \in (0,1),\, N \neq 0,
\end{equation}
for $r>0$ if $K/(N-1) \le 0$ and for $r \in (0,\pi\sqrt{(N-1)/K})$ if $K/(N-1)>0$.
Set also $\bm{\tau}^{(\lambda)}_{K,N}(0):=\lambda$ for all $K,N,\lambda$.
Moreover, when $K/(N-1)>0$, we define for convenience
$\bm{\tau}^{(\lambda)}_{K,N}(r):=\infty$ if $r \ge \pi\sqrt{(N-1)/K}$.

\begin{theorem}[Curvature-dimension condition]\label{th:CD}
For $N \in [n,\infty)$ $($with $N \neq 1$ if $n=1)$ and $K \in \R$,
we have $\Ric_N \ge K$ if and only if $(M,F,\fm)$ satisfies
the \emph{curvature-dimension condition} $\CD(K,N)$ in the sense that,
for any pair of absolutely continuous measures $\mu_0=\rho_0 \fm, \mu_1=\rho_1 \fm \in \cP^2(M)$,
we have for all $\lambda \in (0,1)$
\begin{equation}\label{eq:CD}
S_N(\mu_{\lambda}) \le -\int_{M \times M} \left\{
 \bm{\tau}^{(1-\lambda)}_{K,N}\big( d(x,y) \big) \rho_0(x)^{-1/N}
 +\bm{\tau}^{(\lambda)}_{K,N}\big( d(x,y) \big) \rho_1(y)^{-1/N}  \right\} \pi(dxdy),
\end{equation}
where $(\mu_{\lambda})_{\lambda \in [0,1]} \subset \cP^2(M)$ is
the unique minimal geodesic with respect to $W_2$
and $\pi \in \Pi(\mu_0,\mu_1)$ is the unique $d^2$-optimal coupling.

Similarly, $\Ric_N \ge K$ with $N \in (-\infty,0)$ is equivalent to $\CD(K,N)$ in the sense that
\begin{equation}\label{eq:CDneg}
S_N(\mu_{\lambda}) \le \int_{M \times M} \left\{
 \bm{\tau}^{(1-\lambda)}_{K,N}\big( d(x,y) \big) \rho_0(x)^{-1/N}
 +\bm{\tau}^{(\lambda)}_{K,N}\big( d(x,y) \big) \rho_1(y)^{-1/N}  \right\} \pi(dxdy)
\end{equation}
holds instead of \eqref{eq:CD}$;$
$\Ric_{\infty} \ge K$ is equivalent to $\CD(K,\infty)$ in the sense that
\[ \Ent_{\fm}(\mu_{\lambda}) \le (1-\lambda)\Ent_{\fm}(\mu_0) +\lambda\Ent_{\fm}(\mu_1)
 -\frac{K}{2}(1-\lambda)\lambda W_2^2(\mu_0,\mu_1); \]
and $\Ric_0 \ge K$ is equivalent to $\CD(K,0)$ in the sense that
\begin{equation}\label{eq:CD0}
S_0(\mu_{\lambda}) \le \max \left\{ \esssup_{\supp\pi} \bigg[
 \frac{\bs_{-K}((1-\lambda)d(x,y))}{(1-\lambda)\bs_{-K}(d(x,y))} \rho_0(x) \bigg],
 \esssup_{\supp\pi} \bigg[
 \frac{\bs_{-K}(\lambda d(x,y))}{\lambda \bs_{-K}(d(x,y))} \rho_1(y) \bigg] \right\},
\end{equation}
where the essential supremum is taken with respect to $(x,y) \in \supp\pi$.
\end{theorem}

The equivalence for $N \in [n,\infty]$
(established in \cite{Oint}, see also the survey \cite{Oaspm})
is a generalization to Finsler manifolds of the celebrated result on (weighted) Riemannian manifolds
by \cite{CMS,vRS,StI,StII,LV1,LV2}.
The $N<0$ case was shown in \cite{Oneg}.
The case of $N=0$ is new,
we shall give an outline of the proof of this case after some remarks.

\begin{remark}\label{rm:CD}
(a)
By the definition of $\bm{\tau}^{(\lambda)}_{K,N}$,
on the one hand, \eqref{eq:CD} becomes void if $K>0$ and
\begin{equation}\label{eq:diam}
\pi \left( \{(x,y) \,|\, d(x,y) \ge \pi\sqrt{(N-1)/K}\} \right) >0.
\end{equation}
This, however, never happens thanks to the Bonnet--Myers theorem
available for $N \in [n,\infty)$ and $K>0$.
On the other hand, \eqref{eq:CDneg} becomes trivial if $K<0$ and \eqref{eq:diam} holds.
This means that \eqref{eq:CDneg} gives only a local control when $K<0$.

(b)
Stated in Theorem~\ref{th:CD} is a somewhat simplified version of $\CD(K,N)$,
which is enough for our purpose and still characterizes $\Ric_N \ge K$.
For instance, it is possible and more consistent
to include measures $\mu_0,\mu_1$ with singular parts.

(c)
Let us briefly explain the proof of $\Ric_N \ge K\, \Rightarrow\, \CD(K,N)$
for $N \in (-\infty,0) \cup [n,\infty)$.
Recall that, since $\mu_0 \ll \fm$,
there is a unique $d^2$-optimal coupling $\pi \in \Pi(\mu_0,\mu_1)$
written as $\pi=(\id_M \times T)_{\sharp} \mu_0$ for a measurable map $T:M \lra M$.
Define $T_{\lambda}$ as in \S \ref{ssc:OT},
then $\mu_{\lambda}=(T_{\lambda})_{\sharp}\mu_0$.
The map $T_{\lambda}$ is in fact almost everywhere differentiable and,
under the bound $\Ric_N \ge K$, we have the important concavity inequality:
\begin{equation}\label{eq:Jcon}
\bJ_{\lambda}(x) \ge
 \left\{ \bm{\tau}^{(1-\lambda)}_{K,N}\big( d(x,T(x)) \big)
 +\bm{\tau}^{(\lambda)}_{K,N}\big( d(x,T(x)) \big) \bJ_1(x)^{1/N} \right\}^N
\end{equation}
for $\mu_0$-almost all $x$, where $\bJ_{\lambda}(x)$ is the Jacobian of
$T_{\lambda}$ at $x$ with respect to the measure $\fm$.
Together with the \emph{Jacobian equation} (\emph{Monge--Amper\`e equation}):
\begin{equation}\label{eq:Jeq}
\rho_0(x)=\rho_{\lambda} \big( T_{\lambda}(x) \big) \bJ_{\lambda}(x)
\qquad \text{for $\mu_0$-almost all $x \in M$},
\end{equation}
where $\mu_{\lambda}=\rho_{\lambda} \fm$,
the integration of \eqref{eq:Jcon} yields \eqref{eq:CD} (see \cite{Oint,Oneg} for details).
Conversely, the localization of \eqref{eq:CD} gives the infinitesimal inequality \eqref{eq:Jcon}.

(d)
When $N \in (-\infty,0) \cup [n,\infty]$,
$\CD(0,N)$ means that $S_N$ or $\Ent_{\fm}$ is convex
(in the weak sense) along all $W_2$-geodesics.
For $N=0$, $\CD(0,0)$ implies
\[ S_0(\mu_{\lambda}) \le \max\{ S_0(\mu_0),S_0(\mu_1) \}, \]
namely
\[ \esssup \rho_{\lambda}
 \le \max\left\{ \esssup \rho_0, \esssup \rho_1 \right\}. \]
\end{remark}

\begin{proof}
We give an outline of the proof of Theorem~\ref{th:CD} for $N=0$
using the same notations as \cite[Theorem~4.10]{Oneg}.
We first assume $\Ric_0 \ge K$ and take $\mu_k=\rho_k \fm \in \cP^2(M)$ for $k=0,1$.
Let $T_{\lambda}$ and $\bJ_{\lambda}$ be as in Remark~\ref{rm:CD}(c).
Fix $x \in M$ with $\rho_0(x)>0$ and decompose $\fm$ along
$\eta(\lambda):=T_{\lambda}(x)$ as $\fm|_{\eta}=\e^{-\psi(\lambda)} \vol_{\dot{\eta}}$
(recall Definition~\ref{df:wRic}).
Then we have $\e^{-\psi(0)} \bJ_{\lambda}(x)=h_2(\lambda)^{-1} \e^{\beta(\lambda)}$
for $h_2,\beta$ as in \cite{Oneg}.
By the same calculation as \cite{Oneg}, we find
\begin{align*}
\e^{\beta(\lambda)} &\ge (1-\lambda)\e^{\beta(0)} +\lambda \e^{\beta(1)}, \\
h_2(\lambda) &\le \frac{\bs_{-K}((1-\lambda)d)}{\bs_{-K}(d)} h_2(0)
 +\frac{\bs_{-K}(\lambda d)}{\bs_{-K}(d)} h_2(1),
\end{align*}
where $d:=d(x,T(x))$.
Combining these yields
\begin{align*}
&\e^{-\psi(0)} \bJ_{\lambda}(x)
 = h_2(\lambda)^{-1} \e^{\beta(\lambda)} \\
&\ge \bigg\{ \frac{\bs_{-K}((1-\lambda)d)}{\bs_{-K}(d)} h_2(0)
 +\frac{\bs_{-K}(\lambda d)}{\bs_{-K}(d)} h_2(1) \bigg\}^{-1}
 \{ (1-\lambda)\e^{\beta(0)} +\lambda \e^{\beta(1)} \} \\
&\ge \min \bigg\{ \bigg( \frac{\bs_{-K}((1-\lambda)d)}{\bs_{-K}(d)} h_2(0) \bigg)^{-1}
 (1-\lambda) \e^{\beta(0)},
 \bigg( \frac{\bs_{-K}(\lambda d)}{\bs_{-K}(d)} h_2(1) \bigg)^{-1}
 \lambda \e^{\beta(1)} \bigg\}
\end{align*}
since for $a,b,c,d>0$
\[ \frac{a+b}{c+d} =\frac{c}{c+d} \frac{a}{c} +\frac{d}{c+d} \frac{b}{d}
 \ge \min\bigg\{ \frac{a}{c},\frac{b}{d} \bigg\}. \]
Therefore we have, since  $\bJ_0(x)=1$,
\[ \bJ_{\lambda}(x) \ge \min\bigg\{ 
 \frac{(1-\lambda) \bs_{-K}(d)}{\bs_{-K}((1-\lambda)d)},
 \frac{\lambda \bs_{-K}(d)}{\bs_{-K}(\lambda d)} \bJ_1(x) \bigg\} \]
(this is indeed the limit of \eqref{eq:Jcon} as $N \uparrow 0$).
Together with the Jacobian equation \eqref{eq:Jeq}, this implies
\[ \rho_{\lambda}\big( T_{\lambda}(x) \big)
 =\frac{\rho_0(x)}{\bJ_{\lambda}(x)}
 \le \max\bigg\{ \frac{\bs_{-K}((1-\lambda)d)}{(1-\lambda)\bs_{-K}(d)} \rho_0(x),
 \frac{\bs_{-K}(\lambda d)}{\lambda \bs_{-K}(d)} \rho_1\big( T(x) \big) \bigg\}. \]
Hence we obtain \eqref{eq:CD0}.

To see the converse,
we assume $\CD(K,0)$ and employ the \emph{Brunn--Minkowski inequality} of the form:
\begin{equation}\label{eq:B-M}
\fm(A_{\lambda})
 \ge \min\bigg\{
 \min_{x \in A_0,\, y \in A_1} \frac{(1-\lambda)\bs_{-K}(d(x,y))}{\bs_{-K}((1-\lambda)d(x,y))}
 \fm(A_0),
 \min_{x \in A_0,\, y \in A_1} \frac{\lambda \bs_{-K}(d(x,y))}{\bs_{-K}(\lambda d(x,y))}
 \fm(A_1) \bigg\}
\end{equation}
for Borel measurable sets $A_0,A_1 \subset M$ and $\lambda \in (0,1)$,
where $A_{\lambda}$ is the set consisting of
$\eta(\lambda)$ for minimal geodesics $\eta:[0,1] \lra M$
with $\eta(0) \in A_0$ and $\eta(1) \in A_1$.
Notice that it is enough to consider the case of $0<\fm(A_0),\fm(A_1)<\infty$.
Then one can prove \eqref{eq:B-M} by applying \eqref{eq:CD0} to uniform distributions
$\mu_k:=\fm(A_k)^{-1} \cdot \fm|_{A_k}$ ($k=0,1$):
\begin{align*}
&\fm(A_{\lambda})^{-1}
 \le \fm(\supp \rho_{\lambda})^{-1}
 \le \esssup \rho_{\lambda} \\
&\le \max\bigg\{
 \max_{x \in A_0,\, y \in A_1} \frac{\bs_{-K}((1-\lambda)d(x,y))}{(1-\lambda)\bs_{-K}(d(x,y))}
 \fm(A_0)^{-1},
 \max_{x \in A_0,\, y \in A_1} \frac{\bs_{-K}(\lambda d(x,y))}{\lambda \bs_{-K}(d(x,y))}
 \fm(A_1)^{-1} \bigg\}.
\end{align*}

Fix a unit vector $v \in T_xM$ and let $\eta:(-\delta,\delta) \lra M$
be the geodesic with $\dot{\eta}(0)=v$.
Extend $\dot{\eta}$ to a vector field $V$ around $\Image(\eta)$
such that all integral curves of $V$ are geodesic,
and consider the Riemannian structure $g_V$.
Put $\psi(t):=\Psi(\dot{\eta}(t))$ for $\Psi$ in Definition~\ref{df:wRic},
$a:=-\psi'(0)/n$, and consider the open balls
\[ A_0:=B^V\! \big( \eta(-r),\ve(1+ar) \big),\qquad
 A_1:=B^V\! \big( \eta(r),\ve(1-ar) \big) \]
with respect to $g_V$ for $0<\ve \ll r \ll \delta$.
On the one hand, the asymptotic behavior of $\fm(A_{1/2})$
is controlled in terms of the Ricci curvature and $\psi$ as
\[ \frac{\fm(A_{1/2})}{c_n \ve^n}
 =\e^{-\psi(0)} \bigg( 1+\frac{\Ric(v)}{2} r^2 \bigg) +O(r^3), \]
where $c_n$ is the volume of the unit ball in $\R^n$.
On the other hand, it follows from \eqref{eq:B-M} that
\begin{align*}
\fm(A_{1/2})
&\ge \frac{\bs_{-K}(2r+O(\ve))}{2\bs_{-K}(r+O(\ve))} \min\{ \fm(A_0),\fm(A_1) \} \\
&= \bigg( 1+\frac{K}{2}r^2 +O(r^4) \bigg) \min\{ \fm(A_0),\fm(A_1) \}.
\end{align*}
Note that as $r \downarrow 0$
\begin{align*}
&\fm(A_1)= \e^{-\psi(r)} c_n \ve^n (1-ar)^n +O(\ve^{n+1}) \\
&= c_n \ve^n \e^{-\psi(0)} \bigg[ 1-\{ \psi'(0)+na \} r
 +\frac{1}{2} \{ -\psi''(0) +\psi'(0)^2 +2\psi'(0) na +n(n-1)a^2 \} r^2 \bigg] \\
&\quad +O(r^3) \\
&= c_n \ve^n \e^{-\psi(0)}
 \bigg\{ 1-\frac{1}{2} \bigg( \psi''(0)+\frac{\psi'(0)^2}{n} \bigg) r^2 \bigg\}
 +O(r^3)
\end{align*}
and similarly
\[ \fm(A_0)=c_n \ve^n \e^{-\psi(0)}
 \bigg\{ 1-\frac{1}{2} \bigg( \psi''(0)+\frac{\psi'(0)^2}{n} \bigg) r^2 \bigg\} +O(r^3). \]
Hence we obtain by comparing the coefficients of $r^2$ that
\[ \Ric(v) \ge K-\psi''(0)-\frac{\psi'(0)^2}{n}. \]
Therefore
\[ \Ric_0(v)=\Ric(v)+\psi''(0)+\frac{\psi'(0)^2}{n} \ge K \]
and we complete the proof.
$\qedd$
\end{proof}

We will also use the \emph{measure contraction property} (see \cite{Omcp,StII})
which is weaker and more flexible than the curvature-dimension condition.
The measure contraction property is regarded as a directional Bishop--Gromov inequality,
and makes sense only for $N \in [n,\infty)$.

\begin{theorem}[Measure contraction property]\label{th:MCP}
Assume $\Ric_N \ge K$ for $N \in [n,\infty)$ $($with $N \neq 1$ if $n=1)$
and $K \in \R$.
Then, for any $x \in M$ and a measurable set $A \subset M$ with $0<\fm(A)<\infty$,
we have
\begin{equation}\label{eq:MCP}
\fm(A_{\lambda}) \ge \lambda \inf_{y \in A} \left\{
 \frac{\bs_{K/(N-1)}(\lambda d(x,y))}{\bs_{K/(N-1)}(d(x,y))} \right\}^{N-1} \fm(A)
\end{equation}
for all $\lambda \in (0,1)$,
where $A_{\lambda} \subset M$ is the set consisting of $\eta(\lambda)$
for minimal geodesics $\eta:[0,1] \lra M$ with $\eta(0)=x$ and $\eta(1) \in A$.
\end{theorem}

We remark that the converse (\eqref{eq:MCP} $\Rightarrow$ $\Ric_N \ge K$) holds true only when $N=n$
(see \cite[Remark~5.6]{StII}).
Hence the measure contraction property \eqref{eq:MCP} is strictly weaker than $\CD(K,N)$ in general.

\begin{remark}\label{rm:n=1}
Although Theorems~\ref{th:CD}, \ref{th:MCP} are usually stated for $n \ge 2$,
the case of $n=1$ and $N \in (-\infty,0] \cup (1,\infty]$
is analyzed in the same way more easily (the Ricci curvature term does not appear).
In fact, $\Ric_N \ge K$ (which is read as $\psi''-(\psi')^2/(N-1) \ge K$
with $\fm=\e^{-\psi} \fL^1$) implies the sharp Brunn--Minkowski inequality
via \eqref{eq:CDbd} (see Remark~\ref{rm:CDbd} and \cite{CM2} for details).
From the Brunn--Minkowski inequality one can recover $\Ric_N \ge K$
as well as derive the measure contraction property.
In the excluded case of $n=N=1$,
we have the concrete description of $(M,F,\fm)$ and only $K=0$ makes sense
(see Remark~\ref{rm:wRic}(c)).
\end{remark}

\section{Construction of needle decompositions}\label{sc:needle}%%%%%%%%%%%%%%%%%%%%%%
%%%%%%%%%%%%%%%%

This section is devoted to the construction of a needle decomposition
associated with a $1$-Lipschitz function.
The construction closely follows the strategy developed in optimal transport theory:
we refer to \cite{EG} for the Euclidean case, \cite{FM,Kl} for the Riemannian case,
and \cite{BC,Ca,CM} for the general metric measure setting.

Let the triple $(M,F,\fm)$ be as in the previous section.
Throughout this section, we fix a $1$-Lipschitz function $\varphi:M \lra \R$
in the sense of \eqref{eq:L-Lip}.

\subsection{Transport rays}\label{ssc:trray}%%%%%%%%%%%%%%%%%%%%%%
%%%%%%%%%%%%%%%%

Define
\[ \Gamma_{\varphi}:=\{ (x,y) \in M \times M \,|\, \varphi(y)-\varphi(x)=d(x,y) \}. \]
Clearly $\Gamma_{\varphi}$ is a closed set containing the diagonal set $\{ (x,x) \,|\, x \in M \}$.
A relation with optimal transport theory can be seen in the next lemma.

\begin{lemma}\label{lm:d-cm}
The set $\Gamma_{\varphi}$ is $d$-cyclically monotone.
\end{lemma}

\begin{proof}
For any finite set $\{(x_i,y_i)\}_{i=1}^l \subset \Gamma_{\varphi}$, we have
\[ \sum_{i=1}^l d(x_i,y_{i+1}) \ge \sum_{i=1}^l \{\varphi(y_{i+1})-\varphi(x_i)\}
 =\sum_{i=1}^l \{\varphi(y_i)-\varphi(x_i)\} =\sum_{i=1}^l d(x_i,y_i), \]
where $y_{l+1}:=y_1$.
$\qedd$
\end{proof}

If $(x,y) \in \Gamma_{\varphi}$, then we have,
along any unit speed minimal geodesic $\gamma:[0,d(x,y)] \lra M$ from $x$ to $y$,
\[ \varphi\big( \gamma(t) \big)=\varphi(x)+t \qquad \text{for all}\ t \in [0,d(x,y)]. \]
In particular, $(\gamma(s),\gamma(t)) \in \Gamma_{\varphi}$ for all $0 \le s \le t \le d(x,y)$.
This observation leads us to the following definition.

\begin{definition}[Transport rays]\label{df:ray}
We call a unit speed geodesic $\gamma:\Dom(\gamma) \lra M$ from a closed interval
$\Dom(\gamma) \subset \R$ a \emph{transport ray} associated with $\varphi$ if
\begin{enumerate}[(1)]
\item
$(\gamma(s),\gamma(t)) \in \Gamma_{\varphi}$ for all $s,t \in \Dom(\gamma)$ with $s \le t$;

\item
$\gamma$ cannot be extended to a longer geodesic satisfying the above property (1).
\end{enumerate}
The domain $\Dom(\gamma) \subset \R$ will be taken so as to satisfy $\varphi(\gamma(t))=t$
for all $t \in \Dom(\gamma)$.
If $\Dom(\gamma)$ is a singleton,
then we say that $\gamma$ is a \emph{degenerate} transport ray.
\end{definition}

\begin{example}\label{ex:V}
In the simple example $\varphi(x)=|x|$ on $\R$ (with the standard distance),
we have two transport rays:
$\gamma_1(t)=t$ and $\gamma_2(t)=-t$ for $t \in [0,\infty)$.
\end{example}

Analogous to the set of \emph{strain points} in \cite{Kl},
let us introduce
\begin{equation}\label{eq:M_phi}
M_{\varphi}:=\{ x \in M \,|\, \exists w,y \in M \setminus \{x\}
 \ \text{such that}\ (w,x),\, (x,y) \in \Gamma_{\varphi} \}.
\end{equation}
The following property is easily observed, we state it as a lemma for later use.

\begin{lemma}\label{lm:conca}
Given a triplet $w,x,y \in M$ as in \eqref{eq:M_phi},
let $\gamma_-:[-d(w,x),0] \lra M$ and $\gamma_+:[0,d(x,y)] \lra M$
be any minimal geodesics from $w$ to $x$ and $x$ to $y$, respectively.
Then along the concatenation $\gamma:=\gamma_- \cup \gamma_+:[-d(w,x),d(x,y)] \lra M$
we have
\[ \varphi\big( \gamma(t) \big) -\varphi\big( \gamma(s) \big)=t-s
 \qquad \text{for all}\ s<t. \]
In particular, $\gamma$ is a minimal geodesic from $w$ to $y$,
and $\gamma_-$ and $\gamma_+$ are
unique minimal geodesics from $w$ to $x$ and $x$ to $y$, respectively.
\end{lemma}

On the set $M_{\varphi}$, because of the competition between the $1$-Lipschitz condition
and the defining property of $\Gamma_{\varphi}$, $\varphi$ cannot behave badly.
The next lemma is an analogue of \cite[Lemma~10]{FM}.

\begin{lemma}[Differentiability of $\varphi$ on $M_{\varphi}$]\label{lm:diff}
Given $x \in M_{\varphi}$, let $\gamma:[-\ve,\ve] \lra M$ be a unit speed geodesic
satisfying $\gamma(0)=x$ and $(\gamma(-\ve),x), (x,\gamma(\ve)) \in \Gamma_{\varphi}$.
Then $\varphi$ is differentiable at $x$ with $\nabla \varphi(x)=\dot{\gamma}(0)$.
In particular, such a geodesic $\gamma$ is unique.
\end{lemma}

\begin{proof}
By taking smaller $\ve>0$ if necessary, we can assume that the distance functions
$d(\gamma(-\ve),\cdot)$ and $d(\cdot,\gamma(\ve))$ are smooth in a neighborhood of $x$.
We observe from $\varphi(\gamma(t))=\varphi(x)+t$ that, for any $y \in M$,
\begin{align*}
\varphi(y)-\varphi(x) &=\varphi(y)-\varphi\big( \gamma(-\ve) \big) -\ve 
 \le d\big( \gamma(-\ve),y \big) -\ve, \\
\varphi(y)-\varphi(x) &=\varphi(y)-\varphi\big( \gamma(\ve) \big) +\ve 
 \ge -d\big( y,\gamma(\ve) \big) +\ve.
\end{align*}
Then the claim follows from
\[ \nabla\big[ d\big( \gamma(-\ve),\cdot \big) \big](x) =\nabla\big[ {-}d\big( \cdot,\gamma(\ve) \big) \big](x)
 =\dot{\gamma}(0). \]
$\qedd$
\end{proof}

\subsection{Transport sets}\label{ssc:DTB}%%%%%%%%%%%%%%%%%%%%%%
%%%%%%%%%%%%%%%%

For each point $x \in M$, there are three possibilities:

\begin{enumerate}[(I)]
\item
There is no non-degenerate transport ray containing $x$.
The set of such points $x$ will be denoted by $\bm{D}_{\varphi}$.

\item
There is exactly one non-degenerate transport ray containing $x$.
The set of such points will be called the \emph{transport set} associated with $\varphi$
and denoted by $\bm{T}_{\varphi}$.

\item
There are more than two non-degenerate transport rays containing $x$.
The set of such points will be denoted by $\bm{B}_{\varphi}$.
\end{enumerate}

Clearly $M=\bm{D}_{\varphi} \sqcup \bm{T}_{\varphi} \sqcup \bm{B}_{\varphi}$
and $M_{\varphi} \subset \bm{T}_{\varphi}$ by Lemma~\ref{lm:diff}.
Transport rays give a geodesic foliation of $\bm{T}_{\varphi}$,
that we call the \emph{needle decomposition} associated with $\varphi$.
In the case (III), thanks to Lemma~\ref{lm:diff},
$x$ cannot be an internal point of those transport rays containing $x$.
Precisely, $x$ is either the starting point of all the rays or the terminal point of all the rays
(thus one may say that transport rays branch out from $x$).
Let us introduce the decomposition
$\bm{B}_{\varphi}=\bm{B}_{\varphi}^+ \sqcup \bm{B}_{\varphi}^-$
into starting and terminal points for later convenience, that is,
\begin{equation}\label{eq:B+}
\begin{array}{l}
\bm{B}_{\varphi}^+ :=\{x \in \bm{B}_{\varphi} \,|\, (x,y) \in \Gamma_{\varphi}\ \text{for some}\ y \neq x\},
 \medskip\\
\bm{B}_{\varphi}^- :=\{x \in \bm{B}_{\varphi} \,|\, (w,x) \in \Gamma_{\varphi}\ \text{for some}\ w \neq x\}.
\end{array}
\end{equation}
In the simple example in Example~\ref{ex:V}, we have
$\bm{T}_{\varphi}=\R \setminus \{0\}$ and $\bm{B}_{\varphi}^+=\{0\}$.

\begin{lemma}\label{lm:DTB}
The sets $M_{\varphi}$, $\bm{T}_{\varphi} \sqcup \bm{B}_{\varphi}$,
$\bm{B}_{\varphi}^+$ and $\bm{B}_{\varphi}^-$ are $\sigma$-compact.
In particular, $\bm{D}_{\varphi}$ and $\bm{T}_{\varphi}$ are Borel sets.
\end{lemma}

\begin{proof}
Recall that $\Gamma_{\varphi} \subset M \times M$ is a closed set.
Denoting by $p_k:M^3 \lra M$ ($k=1,2,3$)
the projection to the $k$-th component, we observe that
\begin{align*}
M_{\varphi} &= p_2 \big( \{(w,x,y) \in M^3 \,|\,
 (w,x) \in \Gamma_{\varphi},\, (x,y) \in \Gamma_{\varphi},\,
 w \neq x,\, y \neq x \} \big) \\
&= \bigcup_{i \in \N} p_2 \big( \{(w,x,y) \in M^3 \,|\,
 (w,x) \in \Gamma_{\varphi}, (x,y) \in \Gamma_{\varphi}, d(w,x) \ge i^{-1},
 d(x,y) \ge i^{-1} \} \big)
\end{align*}
is $\sigma$-compact.
We similarly see that
\begin{align*}
\bm{T}_{\varphi} \sqcup \bm{B}_{\varphi}
&=\bigcup_{k=1}^2 p_k\big( \{(x,y) \in \Gamma_{\varphi} \,|\, x \neq y\} \big), \\
\bm{B}_{\varphi}^+
&=p_1\big( \{(x,y,z) \in M^3 \,|\, (x,y),(x,z) \in \Gamma_{\varphi},\, d(x,y)=d(x,z)>0,\, y \neq z\} \big),
\end{align*}
and $\bm{B}_{\varphi}^-$ are all $\sigma$-compact.
$\qedd$
\end{proof}

The set $\bm{D}_{\varphi}$ is in general large,
we have even $\bm{D}_{\varphi}=M$ if $\varphi$ is ($1-\ve$)-Lipschitz
for some $\ve>0$.
The transport set $\bm{T}_{\varphi}$ is our main object,
being a full measure set in the situation we consider in applications
(by virtue of Proposition~\ref{pr:mean0}).
One can see that $\fm(\bm{B}_{\varphi})=0$ always holds true
with the help of the following useful lemma, which has played crucial roles in \cite{Ca,Ca2,CM}.

\begin{lemma}\label{lm:d^2}
Suppose that a subset $\Xi \subset \Gamma_{\varphi}$ satisfies
$\sup_{(x,y) \in \Xi}\varphi(x) \le \inf_{(x,y) \in \Xi}\varphi(y)$ and
\begin{equation}\label{eq:d^2}
\{\varphi(x_2)-\varphi(x_1)\} \{\varphi(y_2)-\varphi(y_1)\} \ge 0\qquad
 \text{for all}\ (x_1,y_1),(x_2,y_2) \in \Xi.
\end{equation}
Then $\Xi$ is $d^2$-cyclically monotone.
\end{lemma}

\begin{proof}
We first see that the set
\[ \Xi':=\left\{ \big( \varphi(x),\varphi(y) \big) \,\big|\, (x,y) \in \Xi \right\} \subset \R \times \R \]
is $|\cdot|^2$-cyclically monotone by induction,
where $|\cdot|$ is the standard distance.
For any pair $(x_1,y_1),(x_2,y_2) \in \Xi$,
the hypothesis \eqref{eq:d^2} immediately yields (with $y_3:=y_1$)
\[ \sum_{i=1}^2 |\varphi(y_{i+1})-\varphi(x_i)|^2
 -\sum_{i=1}^2 |\varphi(y_i)-\varphi(x_i)|^2
 =2\{\varphi(x_2)-\varphi(x_1)\} \{\varphi(y_2)-\varphi(y_1)\} \ge 0. \]
Suppose that the claim holds true for any set consisting of $(l-1)$ elements in $\Xi'$,
and take arbitrary $\{(x_i,y_i)\}_{i=1}^l \subset \Xi$.
It follows from \eqref{eq:d^2} that
$\varphi(x_i)<\varphi(x_j)$ only if $\varphi(y_i) \le \varphi(y_j)$.
Hence we can assume without loss of generality that
$\varphi(x_1)=\min_i \varphi(x_i)$ as well as $\varphi(y_1)=\min_i \varphi(y_i)$.
Then, putting $s_i:=\varphi(x_i)$ and $t_i:=\varphi(y_i)$ for simplicity,
we have (with $t_{l+1}:=t_1$)
\begin{align*}
\sum_{i=1}^l |t_{i+1}-s_i|^2 -\sum_{i=1}^l |t_i-s_i|^2
&\ge \{ (t_2-s_1)^2 +(t_1-s_l)^2 -(t_2-s_l)^2 \} -(t_1-s_1)^2 \\
&=2(t_2-t_1)(s_l-s_1) \ge 0,
\end{align*}
where we applied the claim to $\{(s_i,t_i)\}_{i=2}^l$ in the inequality.
Therefore $\Xi'$ is $|\cdot|^2$-cyclically monotone.

Now, for any $\{(x_i,y_i)\}_{i=1}^l \subset \Xi$,
we observe from $\sup_{(x,y) \in \Xi}\varphi(x) \le \inf_{(x,y) \in \Xi}\varphi(y)$
that $0 \le \varphi(y_{i+1})-\varphi(x_i) \le d(x_i,y_{i+1})$.
Hence we obtain
\[ \sum_{i=1}^l d^2(x_i,y_i)
 =\sum_{i=1}^l \{ \varphi(y_i)-\varphi(x_i) \}^2
 \le \sum_{i=1}^l \{ \varphi(y_{i+1})-\varphi(x_i) \}^2
 \le \sum_{i=1}^l d^2(x_i,y_{i+1}), \]
where $y_{l+1}:=y_1$.
This completes the proof.
$\qedd$
\end{proof}

We remark that the assumption $\sup_{(x,y) \in \Xi}\varphi(x) \le \inf_{(x,y) \in \Xi}\varphi(y)$
is unnecessary if $F$ is reversible.
The next proposition, an analogue of \cite[Proposition~4.5]{Ca},
is an interesting application of a basic fact in optimal transport theory.
For $x \in M$ and $r>0$, let
\[ B^+(x,r):=\{y \in M \,|\, d(x,y)<r\}, \qquad B^-(x,r):=\{y \in M \,|\, d(y,x)<r\} \]
be the \emph{forward} and \emph{backward open balls} with center $x$ and radius $r$.

\begin{proposition}\label{pr:Bnull}
We have $\fm(\bm{B}_{\varphi})=0$.
\end{proposition}

\begin{proof}
Recalling the decomposition $\bm{B}_{\varphi}=\bm{B}_{\varphi}^+ \sqcup \bm{B}_{\varphi}^-$
in \eqref{eq:B+}, we suppose in contradiction that $\fm(\bm{B}_{\varphi}^+)>0$.
Then the case of $\fm(\bm{B}_{\varphi}^-)>0$ is covered as well by considering
the function $-\varphi$ which is $1$-Lipschitz with respect to the reverse Finsler structure $\rev{F}$.

By the definition of $\bm{B}_{\varphi}^+$, for each $x \in \bm{B}_{\varphi}^+$,
there are (at least) two distinct non-degenerate transport rays emanating from $x$.
Let us take $r \gg \ve>0$ for which the set $A(r,\ve) \subset \bm{B}_{\varphi}^+$
consisting of $x$ such that
there are transport rays $\gamma_x^1,\gamma_x^2$ emanating from $x$ with
\[ \Dom(\eta_x^1) \cap \Dom(\eta_x^2) \supset [0,2r], \qquad
 \min\big\{ d\big( \eta_x^1(2r),\eta_x^2(2r) \big),d\big( \eta_x^2(2r),\eta_x^1(2r) \big) \big\} \ge \ve r \]
has a positive measure,
where we set $\eta_x^k(t):=\gamma_x^k(\varphi(x)+t)$ for $k=1,2$ for brevity.
Notice that $A(r,\ve)$ is a closed set in $\bm{B}_{\varphi}^+$.
Thanks to the selection theorem stated below (Theorem~\ref{th:select}),
one can choose $\gamma_x^k$ in such a way that $T^k(x):=\eta_x^k(2r)$
is a Borel map on $A(r,\ve)$ for $k=1,2$.
Precisely, we consider the map $\Upsilon:A(r,\ve) \lra 2^M$
sending $x$ to the nonempty set $\{\gamma(\varphi(x)+2r)\}_{\gamma}$,
where $\gamma$ runs over all transport rays emanating from $x$ whose domains include $[0,2r]$.
For any open set $U \subset M$,
\[ \{x \in A(r,\ve) \,|\, \Upsilon(x) \cap U \neq \emptyset\}
 =p_1\big( \{(x,y) \in \Gamma_{\varphi} \,|\, x \in A(r,\ve),\, y \in U,\, d(x,y)=2r\} \big) \]
is a Borel set.
Thus we can select a Borel map $T^1:A(r,\ve) \lra M$ such that
$T^1(x) \in \Upsilon(x)$ for all $x \in A(r,\ve)$.
One can similarly obtain $T^2$ from
\[ \Upsilon'(x):=\Upsilon(x) \setminus \big( B^+(T^1(x),\ve r) \cup B^-(T^1(x),\ve r) \big)
 \neq \emptyset, \]
since
\begin{align*}
&\{x \in A(r,\ve) \,|\, \Upsilon'(x) \cap U \neq \emptyset\} \\
&=p_1\bigg( \bigg\{(x,y) \in \Gamma_{\varphi} \,\bigg|
 \begin{array}{l}
 x \in A(r,\ve),\, y \in U,\, d(x,y)=2r, \smallskip\\
 \min\big\{ d(T^1(x),y),d(y,T^1(x)) \big\} \ge \ve r
 \end{array} \bigg\} \bigg)
\end{align*}
is Borel.

Now we fix $x_0 \in A(r,\ve)$ such that
\[ A_{x_0}(r,\ve):=A(r,\ve) \cap B^+(x_0,r) \cap B^-(x_0,r) \]
has a positive measure.
Since $\varphi$ is $1$-Lipschitz and
$\varphi(T^1(x))=\varphi(T^2(x))=\varphi(x)+2r$ for all $x \in A_{x_0}(r,\ve)$, we find
\[ \sup_{x \in A_{x_0}(r,\ve)} \varphi(x) \le \varphi(x_0)+r
 \le \inf_{x \in A_{x_0}(r,\ve)} \varphi(x)+2r
 =\inf_{x \in A_{x_0}(r,\ve)} \varphi\big( T^k(x) \big) \]
for $k=1,2$.
Hence the set
\[ \Xi:=\big\{ \big(x,T^k(x) \big) \,\big|\, x \in A_{x_0}(r,\ve),\, k=1,2 \big\} \]
is $d^2$-cyclically monotone by Lemma~\ref{lm:d^2}.
It follows from Theorem~\ref{th:cymo} that the coupling
\[ \pi:=\frac{1}{2}(\id_M \times T^1)_{\sharp}
 \left( \frac{\fm|_{A_{x_0}(r,\ve)}}{\fm(A_{x_0}(r,\ve))} \right)
 +\frac{1}{2}(\id_M \times T^2)_{\sharp}
 \left( \frac{\fm|_{A_{x_0}(r,\ve)}}{\fm(A_{x_0}(r,\ve))} \right) \]
is $d^2$-optimal.
This is, however, a contradiction since $\Xi$ cannot be represented
as the graph of any map (recall the Brenier--McCann theorem in \S \ref{ssc:OT}).
Therefore we conclude $\fm(\bm{B}_{\varphi}^+)=0$.
$\qedd$
\end{proof}

We have used in the above proof the following special case
of the classical \emph{selection theorem} (see \cite[Theorem in p.~398]{KR}).

\begin{theorem}\label{th:select}
Let $X,Y$ be metric spaces and suppose that $Y$ is complete and separable.
If a map $\Upsilon:X \lra 2^Y$ satisfies that
\[ \{x \in X \,|\, \Upsilon(x) \cap U \neq \emptyset\}\ \text{is Borel for any open set}\ U \subset Y, \]
then there exists a Borel map $T:X \lra Y$ such that $T(x) \in \Upsilon(x)$ for any $x \in X$.
\end{theorem}

\subsection{Disintegration}\label{ssc:disint}%%%%%%%%%%%%%%%%%%%%%%
%%%%%%%%%%%%%%%%

Denote the set of non-degenerate transport rays
$\gamma:\Dom(\gamma) \lra \R$ by $\bm{R}_{\varphi}$.
Observe that
\[ \bigcup_{\gamma  \in \bm{R}_{\varphi}} \Image(\gamma)
 =\bm{T}_{\varphi} \sqcup \bm{B}_{\varphi}. \]
We shall give another interpretation of $\bm{R}_{\varphi}$ as a quotient of $\bm{T}_{\varphi}$.
This will lead us to a disintegration of $\fm|_{\bm{T}_{\varphi}}$ with respect to $\bm{R}_{\varphi}$.

\begin{lemma}\label{lm:T_phi}
The relation $\sim$ on $\bm{T}_{\varphi}$ defined by
\[ x \sim y\quad \text{if}\ (x,y) \in \Gamma_{\varphi}\ \text{or}\ (y,x) \in \Gamma_{\varphi} \]
is an equivalence relation.
Moreover, the map $\Theta:\bm{T}_{\varphi}/{\sim} \lra \bm{R}_{\varphi}$
sending $[\gamma(t)]$ to $\gamma$ for each $\gamma \in \bm{R}_{\varphi}$
with $\gamma(t) \in \bm{T}_{\varphi}$ is well-defined and bijective.
\end{lemma}

\begin{proof}
It is obvious that the relation $\sim$ is symmetric and $x \sim x$ for all $x \in \bm{T}_{\varphi}$.

Given $x \in \bm{T}_{\varphi}$, there is a unique transport ray
$\gamma:\Dom(\gamma) \lra M$ passing through $x$.
Clearly every $y \in \Image(\gamma) \cap \bm{T}_{\varphi}$ enjoys $x \sim y$.
Conversely, if $y \in \bm{T}_{\varphi}$ satisfies $x \sim y$ with $(x,y) \in \Gamma_{\varphi}$,
then any minimal geodesic from $x$ to $y$ needs to be a part of $\gamma$
by the definition of $\bm{T}_{\varphi}$
(thus, in particular, a minimal geodesic from $x$ to $y$ is unique).
In the other case of $(y,x) \in \Gamma_{\varphi}$,
we similarly see that a unique minimal geodesic from $y$ to $x$ is a part of $\gamma$.
This shows that $\sim$ is transitive and $\Theta$ is bijective.
$\qedd$
\end{proof}

We equip $\bm{R}_{\varphi}$ with the quotient topology induced from
the identification with $\bm{T}_{\varphi}/{\sim}$ via $\Theta$.
Then $\bm{R}_{\varphi}$ is $\sigma$-compact
since $M_{\varphi}$ is $\sigma$-compact (Lemma~\ref{lm:DTB})
and $p(M_{\varphi})=\bm{R}_{\varphi}$,
where $p:\bm{T}_{\varphi} \lra \bm{R}_{\varphi}$ is the projection.
We remark that $\sim$ is not an equivalence relation of $\bm{T}_{\varphi} \sqcup \bm{B}_{\varphi}$.
Consider the example $\varphi(x)=|x|$ in Example~\ref{ex:V}
and observe $-1 \sim 0$, $0 \sim 1$, but $-1 \not\sim 1$.

Following \cite{BC} (see also \cite{Ca,CM}),
we shall disintegrate $\fm|_{\bm{T}_{\varphi}}$ along $\bm{T}_{\varphi}/{\sim}$,
via a Borel map $\sigma:\bm{T}_{\varphi} \lra \bm{T}_{\varphi}$ satisfying
\begin{equation}\label{eq:sigma}
x \sim \sigma(x)\ \text{for all}\ x \in \bm{T}_{\varphi},
 \qquad \text{and}\ \sigma(x)=\sigma(y)\ \text{if}\ x \sim y.
\end{equation}
Such a map $\sigma$ is given again by the selection theorem as follows.

\begin{lemma}\label{lm:sigma}
There exists a Borel map $\sigma:\bm{T}_{\varphi} \lra \bm{T}_{\varphi}$ satisfying \eqref{eq:sigma}.
\end{lemma}

\begin{proof}
Theorem~\ref{th:select} applies to $\Upsilon:\bm{T}_{\varphi} \lra 2^M$
defined as $\Upsilon(x):=\Image(\gamma) \cap \bm{T}_{\varphi}$,
where $\gamma$ is the unique transport ray containing $x$.
Indeed, for any open set $U \subset M$,
\[ \{x \in \bm{T}_{\varphi} \,|\, \Upsilon(x) \cap U \neq \emptyset\}
 =p_1 \big( \{(x,y) \in \bm{T}_{\varphi}^2 \,|\, y \in U,\, x \sim y\} \big) \]
is a Borel set.
$\qedd$
\end{proof}

We fix $\sigma$ as in Lemma~\ref{lm:sigma} from here on
($\sigma$ will play a role also in \S \ref{sc:prop}).
By identifying $\sigma(\bm{T}_{\varphi}) \subset \bm{T}_{\varphi}$ with $\bm{R}_{\varphi}$
(via $\bm{T}_{\varphi}/{\sim}$),
\[ \fv:=\sigma_{\sharp}(\fm|_{\bm{T}_{\varphi}}) \]
becomes a Borel regular measure on $\bm{R}_{\varphi}$
with $\fv(\bm{R}_{\varphi})=\fm(\bm{T}_{\varphi})$.
Then the \emph{disintegration theorem}
(see \cite[III.70--73]{DM} or \cite[Theorem~5.3.1]{AGSbook}) gives
a $\fv$-almost everywhere uniquely determined family of Borel probability measures
$\{\bm{\mu}_{\gamma}\}_{\gamma \in \bm{R}_{\varphi}} \subset \cP(M)$ such that
$\bm{\mu}_{\gamma}(\Image(\gamma) \cap \bm{T}_{\varphi})=1$ and
\[ \int_{\bm{T}_{\varphi}} f \,d\fm
 =\int_{\bm{R}_{\varphi}} \int_{\Image(\gamma)} f(x) \,\bm{\mu}_{\gamma}(dx) \,\fv(d\gamma) \]
for all Borel integrable functions $f$ on $\bm{T}_{\varphi}$.
We will use the slightly rewritten form:
\begin{equation}\label{eq:disint}
\int_{\bm{T}_{\varphi}} f \,d\fm
 =\int_{\bm{R}_{\varphi}} \int_{\Dom(\gamma)} f\big( \gamma(t) \big) \,\bm{\mu}_{\gamma}(dt) \,\fv(d\gamma),
\end{equation}
by regarding $\bm{\mu}_{\gamma}$ as a measure on the interval $\Dom(\gamma) \subset \R$.
Further qualitative and quantitative properties of the measures $\bm{\mu}_{\gamma}$
will be discussed in \S \ref{ssc:qual} and \S \ref{ssc:quan}, respectively.

\section{Needle decompositions conditioned by mean-zero functions}\label{sc:mean0}%%%%%%%%%%%%
%%%%%%%%%%%%%%%%

Before explaining further properties of disintegrated measures,
we present in this section an important situation to which we apply our construction.
It also reveals a deep connection between our construction and optimal transport theory.

We fix an $\fm$-integrable function $f:M \lra \R$ such that $\int_M f \,d\fm=0$ (mean-zero) and
\[ \int_M |f(x)| \{d(x_0,x)+d(x,x_0)\} \,\fm(dx) <\infty \]
for some (hence all) $x_0 \in M$.
Consider the following maximization problem:
\begin{equation}\label{eq:MK}
\text{Find a $1$-Lipschitz function $\varphi$ maximizing}\ \int_M f\phi \,d\fm\
\text{among all $\phi \in \Lip_1(M)$}.
\end{equation}
Note that $f\phi$ is $\fm$-integrable since, given $x_0 \in M$,
\begin{align*}
\int_M |f\phi| \,d\fm
&\le \int_M |f(x)| \big( |\phi(x_0)|+\max\{d(x_0,x),d(x,x_0) \} \big) \,\fm(dx) \\
&= |\phi(x_0)| \int_M |f| \,d\fm +\int_M |f(x)| \max\{d(x_0,x),d(x,x_0) \} \,\fm(dx) \\
&< \infty.
\end{align*}
The condition $\int_M f \,d\fm=0$ yields that $\int_M f\phi \,d\fm =\int_M f\{\phi-\phi(x_0)\} \,d\fm$,
hence we can restrict ourselves to $\phi \in \Lip_1(M)$ with $\phi(x_0)=0$
for some fixed point $x_0 \in M$ in \eqref{eq:MK}.
Thus we also find $\sup_{\phi \in \Lip_1(M)} \int_M f\phi \,d\fm<\infty$.

When $f$ is given as the difference of the densities of two absolutely continuous probability measures
($f=\rho_1-\rho_0$ with $\mu_k=\rho_k \fm \in \cP^1(M)$),
\eqref{eq:MK} is nothing but the dual formulation of the Monge--Kantorovich problem
for the cost function $c(x,y):=d(x,y)$ (recall \eqref{eq:KR}).

One can find a solution to \eqref{eq:MK} by a simple application of the Ascoli--Arzel\`a theorem.

\begin{lemma}\label{lm:MK}
There exists a $1$-Lipschitz function $\varphi:M \lra \R$ which is a solution to \eqref{eq:MK}.
\end{lemma}

\begin{proof}
Take a sequence $\{\phi_i\}_{i \in \N} \subset \Lip_1(M)$ satisfying
\[ \lim_{i \to \infty} \int_M f\phi_i \,d\fm=\sup_{\phi \in \Lip_1(M)}\int_M f\phi \,d\fm \]
and $\phi_i(x_0)=0$ for some $x_0 \in M$ and all $i \in \N$.
By the Ascoli--Arzel\`a theorem and the diagonal argument,
a subsequence of $\{\phi_i\}_{i \in \N}$ converges to some $\varphi \in \Lip_1(M)$
uniformly on each compact set.
Thanks to $-d(x,x_0) \le \phi_i(x) \le d(x_0,x)$ for all $i \in \N$ and $x \in M$,
one can use the dominated convergence theorem to see
\[ \int_M f\phi \,d\fm =\lim_{i \to \infty} \int_M f\phi_i \,d\fm=\sup_{\phi \in \Lip_1(M)}\int_M f\phi \,d\fm. \]
$\qedd$
\end{proof}

In the remainder of this section,
we fix a $1$-Lipschitz function $\varphi$ given by Lemma~\ref{lm:MK}.
Along the lines of \cite{EG} (see also \cite[\S 4]{Kl}),
we shall show that $f$ is mean-zero along almost every transport ray associated with $\varphi$.
Given a Borel set $A \subset M$, we define
\[ S(A):=\{x \in M \,|\,
 (x,y) \in \Gamma_{\varphi}\ \text{or}\ (y,x) \in \Gamma_{\varphi}\ \text{for some}\ y \in A\}. \]
Clearly $A \subset S(A)$ and $S(A)$ is a Borel set.
We say that $A$ is a \emph{saturated set} associated with $\varphi$ if $A=S(A)$.
Notice that a singleton $\{x\}$ is a saturated set if $x \in \bm{D}_{\varphi}$.

\begin{lemma}\label{lm:balance}
Take a compact set $Z \subset M$ and $\delta>0$.
\begin{enumerate}[{\rm (i)}]
\item
The function
\[ \varphi_{\delta}(x) :=\inf_{y \in M} \{\varphi(y)+d(y,x)-\delta \cdot \chi_Z(y)\},\qquad x \in M, \]
satisfies $0 \le \varphi(x)-\varphi_{\delta}(x) \le \delta$ for all $x \in M$,
where $\chi_Z$ is the characteristic function of $Z$.

\item
The limit
\[ \Phi(x):=\lim_{\delta \downarrow 0} \frac{\varphi(x)-\varphi_{\delta}(x)}{\delta} \]
exists in $[0,1]$ at all $x \in M$, and we have
$\Phi \equiv 1$ on $Z$ and $\Phi \equiv 0$ on $M \setminus S(Z)$.
\end{enumerate}
\end{lemma}

\begin{proof}
(i)
On the one hand, choosing $y=x$ in the definition of $\varphi_{\delta}$,
we find
\[ \varphi_{\delta}(x) \le \varphi(x)-\delta \cdot \chi_Z(x) \le \varphi(x). \]
On the other hand, since $\varphi$ is $1$-Lipschitz,
\[ \varphi_{\delta}(x) \ge \inf_{y \in M} \{\varphi(y)+d(y,x)\} -\delta \ge \varphi(x)-\delta. \]

(ii)
We first observe that, since $\varphi(x)-\varphi(y)-d(y,x) \le 0$,
\[ \frac{\varphi(x)-\varphi_{\delta}(x)}{\delta}
 =\sup_{y \in M} \bigg\{ \frac{\varphi(x)-\varphi(y)-d(y,x)}{\delta}+\chi_Z(y) \bigg\} \in [0,1] \]
is non-increasing as $\delta \downarrow 0$.
Hence $\Phi(x)$ is well-defined for all $x \in M$.

If $x \in Z$, then the definition of $\varphi_{\delta}$ and (i) above imply
$\varphi_{\delta}(x) \le \varphi(x)-\delta \le \varphi_{\delta}(x)$.
Therefore $\varphi_{\delta}(x)=\varphi(x)-\delta$ for all $\delta>0$ and $\Phi(x)=1$.

Let $x \in M \setminus S(Z)$.
Then, for any $y \in Z$, it follows from the definition of $S(Z)$ that
$x$ and $y$ are not on the same transport ray.
This implies $\varphi(y)+d(y,x)>\varphi(x)$, and the compactness of $Z$
gives a positive constant $\delta_x>0$ such that
\[ \varphi(y)+d(y,x) \ge \varphi(x)+\delta_x \]
for all $y \in Z$.
Hence $\varphi_{\delta}(x)=\varphi(x)$ for all $\delta \le \delta_x$, and $\Phi(x)=0$.
$\qedd$
\end{proof}

\begin{lemma}\label{lm:mean0}
For any saturated set $A \subset M$ associated with $\varphi$, we have
\[ \int_A f \,d\fm=0. \]
\end{lemma}

\begin{proof}
We first show that $\int_A f \,d\fm \ge 0$.
For arbitrary $\ve>0$, since $\fm$ is Borel regular,
there is a compact set $Z \subset A$ such that $\int_{A \setminus Z} |f| \,d\fm<\ve$.
Then $S(Z) \subset S(A)=A$.

Take $\delta>0$ and let $\varphi_{\delta}$ be as in Lemma~\ref{lm:balance}.
Note that the function $x \longmapsto d(y,x)$ is $1$-Lipschitz for every fixed $y \in M$.
This implies that the function
\[ x\ \longmapsto\ \varphi(y)+d(y,x)-\delta \cdot \chi_Z(y) \]
is $1$-Lipschitz for every $y \in M$, and hence $\varphi_{\delta} \in \Lip_1(M)$.
Then by the choice of $\varphi$ (Lemma~\ref{lm:MK}) we have
\[ \int_M f \frac{\varphi-\varphi_{\delta}}{\delta} \,d\fm \ge 0
 \qquad \text{for all}\ \delta>0. \]
Letting $\delta \downarrow 0$, we obtain by the dominated convergence theorem
and Lemma~\ref{lm:balance}(ii) that
\[ 0 \le \int_M f\Phi \,d\fm =\int_{S(Z)} f\Phi \,d\fm
 =\int_{S(Z) \setminus Z} f\Phi \,d\fm +\int_Z f \,d\fm. \]
Combining this with $0 \le \Phi \le 1$, we have
\[ \int_Z f \,d\fm \ge -\int_{S(Z) \setminus Z} f\Phi \,d\fm
 \ge -\int_{S(Z) \setminus Z} |f| \,d\fm >-\ve. \]
Since $\ve>0$ was arbitrary, this yields $\int_A f \,d\fm \ge 0$.

To see $\int_A f \,d\fm \le 0$,
we simply apply the above claim to the functions $-f$ and $-\varphi$,
the latter is $1$-Lipschitz with respect to the reverse Finsler structure $\rev{F}$.
$\qedd$
\end{proof}

\begin{proposition}\label{pr:mean0}
We have
\[ \int_{\Dom(\gamma)} f \big( \gamma(t) \big) \,\bm{\mu}_{\gamma}(dt)=0 \]
for $\fv$-almost every $\gamma \in \bm{R}_{\varphi}$.
Moreover, $f \equiv 0$ $\fm$-almost everywhere on $\bm{D}_{\varphi}$.
\end{proposition}

\begin{proof}
Recall the disintegration \eqref{eq:disint}:
\[ \int_{\bm{T}_{\varphi}} f \,d\fm
 =\int_{\bm{R}_{\varphi}} \int_{\Dom(\gamma)} f\big( \gamma(t) \big) \,\bm{\mu}_{\gamma}(dt) \,\fv(d\gamma). \]
For any Borel set $B \subset \bm{R}_{\varphi}$,
$p^{-1}(B) \subset \bm{T}_{\varphi}$ is a saturated set
up to an $\fm$-negligible set contained in $\bm{B}_{\varphi}$,
where $p:\bm{T}_{\varphi} \lra \bm{R}_{\varphi}$ is the projection.
Then it follows from Lemma~\ref{lm:mean0} that
\[ \int_B \int_{\Dom(\gamma)} f\big( \gamma(t) \big) \,\bm{\mu}_{\gamma}(dt) \,\fv(d\gamma)
 =\int_{\sigma^{-1}(B)} f \,d\fm=0, \]
and hence $\int_{\Dom(\gamma)} f \circ \gamma \,d\bm{\mu}_{\gamma}=0$
for $\fv$-almost every $\gamma \in \bm{R}_{\varphi}$.

The second assertion also follows from Lemma~\ref{lm:mean0}
since any Borel set $A \subset \bm{D}_{\varphi}$ is a saturated set.
$\qedd$
\end{proof}

\section{Properties of disintegrated measures}\label{sc:prop}%%%%%%%%%%%%%%%%%%%%%%
%%%%%%%%%%%%%%%%

Coming back to a general $1$-Lipschitz function $\varphi:M \lra \R$,
we investigate in this section properties of the disintegrated measures
$\bm{\mu}_{\gamma}$, $\gamma \in \bm{R}_{\varphi}$, given in \S \ref{ssc:disint}.

\subsection{Ray maps}\label{ssc:raymap}%%%%%%%%%%%%%%%%%%%
%%%%%%%%%%%%%%%%

Along \cite[\S 4]{BC}, we introduce the following.

\begin{definition}[Ray maps]\label{df:raymap}
Define the subset $\Dom(\cG) \subset \sigma(\bm{T}_{\varphi}) \times \R$
($\subset \bm{T}_{\varphi} \times \R$) by
\begin{align}
\Dom(\cG):= &\{(x,t) \,|\,
 (x,y) \in \Gamma_{\varphi},\, t=d(x,y)\ \text{for some}\ y \in \bm{T}_{\varphi} \}
 \nonumber\\
&\cup \{(x,t) \,|\,
 (w,x) \in \Gamma_{\varphi},\, t=-d(w,x)\ \text{for some}\ w \in \bm{T}_{\varphi} \},
 \label{eq:DomG}
\end{align}
and the \emph{ray map} $\cG:\Dom(\cG) \lra \bm{T}_{\varphi}$ by
$\cG(x,t):=y$ in the first set of \eqref{eq:DomG}
and $\cG(x,t):=w$ in the second set.
\end{definition}

Recall that,
if $(x,y) \in \Gamma_{\varphi}$ with $x \in \bm{T}_{\varphi}$ and $y \neq x$,
then there is a transport ray $\gamma$ passing through $x$ and $y$,
and $\gamma$ is the unique transport ray containing $x$.
Thus the map $\cG$ is well-defined.
Clearly $\sigma(\bm{T}_{\varphi}) \times \{0\} \subset \Dom(\cG)$
and $\cG$ is a bijective map from $\Dom(\cG)$ to $\bm{T}_{\varphi}$.
Since the graph of $\cG$
\begin{align*}
&\{(x,t,y) \in \sigma(\bm{T}_{\varphi}) \times \R \times \bm{T}_{\varphi} \,|\,
 (x,t) \in \Dom(\cG),\, y=\cG(x,t)\} \\
&= \{(x,t,y) \,|\, (x,y) \in \Gamma_{\varphi},\, t=d(x,y)\}
 \cup \{(x,t,w) \,|\, (w,x) \in \Gamma_{\varphi},\, t=-d(w,x)\}
\end{align*}
is a Borel set, $\cG$ and $\cG^{-1}$ are Borel maps.
Moreover, $\Dom(\cG)$ is convex in the $\R$-direction in the sense that
$(x,s),(x,t) \in \Dom(\cG)$ with $s<t$ implies $\{x\} \times [s,t] \subset \Dom(\cG)$.
We also observe from Lemma~\ref{lm:diff} that
\[ \del_r \cG(x,r) =\nabla\varphi \big( \cG(x,r) \big)
 \qquad \text{for all}\ r \in (s,t). \]
Hence $\cG$ is regarded as the (ascending) \emph{gradient flow}
of $\varphi$ on $\bm{T}_{\varphi}$.

The convexity of $\Dom(\cG)$ in the $\R$-direction readily implies the following property
that will be useful in the sequel
and can be compared with the decompositions into ray clusters in \cite[\S 3.2]{Kl}.

\begin{lemma}\label{lm:Rconv}
There exists a countable family $\{Z_i\}_{i \in \N}$ of compact subsets of $\bm{T}_{\varphi}$
together with
$\{(a_i,b_i)\}_{i \in \N} \subset \R \times \R$ such that $a_i \le 0 \le b_i$ and
\[ \sigma(Z_i) \times [a_i,b_i] \subset \Dom(\cG), \qquad
 (\fv \times \fL^1)\bigg( \Dom(\cG) \setminus \bigcup_{i \in \N}
 \big(\sigma(Z_i) \times [a_i,b_i] \big) \bigg)=0, \]
where $\fL^1$ denotes the $1$-dimensional Lebesgue measure on $\R$.
\end{lemma}

\begin{proof}
Take dense sets $\{a_i\}_{i \in \N} \subset (-\infty,0]$
and $\{b_j\}_{j \in \N} \subset [0,\infty)$ with $a_1=b_1=0$,
and consider the intervals $I_{i,j}:=[a_i,b_j]$.
For every $x \in \sigma(\bm{T}_{\varphi})$
and the transport ray $\gamma_x \in \bm{R}_{\varphi}$ passing through $x$,
there are $i,j \in \N$ such that
\[ \gamma_x^{-1}(\bm{T}_{\varphi})
 \supset [\varphi(x)+a_i,\varphi(x)+b_j]. \]
Let $A_{i,j}$ be the set of points $x \in \sigma(\bm{T}_{\varphi})$
satisfying $\gamma_x^{-1}(\bm{T}_{\varphi}) \supset [\varphi(x)+a_i,\varphi(x)+b_j]$.
We observe from the convexity of $\Dom(\cG)$ in the $\R$-direction that
\[ A_{i,j}=\{x \in \sigma(\bm{T}_{\varphi}) \,|\, (x,a_i),(x,b_j) \in \Dom(\cG)\} \]
which is a Borel set.
Moreover, we have
\[ A_{i,j} \times I_{i,j} \subset \Dom(\cG),\qquad
 (\fv \times \fL^1)\bigg( \Dom(\cG) \setminus \bigcup_{i,j \in \N}
 (A_{i,j} \times I_{i,j}) \bigg)=0. \]
Since $\fm$ is Borel regular,
we can choose compact sets $Z_{i,j}^k \subset \sigma^{-1}(A_{i,j}) \subset \bm{T}_{\varphi}$,
$k \in \N$, with $\fv(A_{i,j} \setminus \bigcup_{k \in \N} \sigma(Z_{i,j}^k))=0$.
Then the family $\{\sigma(Z_{i,j}^k) \times [a_i,b_j]\}_{i,j,k \in \N}$
satisfies the desired properties.
$\qedd$
\end{proof}

\subsection{Qualitative properties derived from MCP}\label{ssc:qual}%%%%%%%%%%%%%
%%%%%%%%%%%%%%%%

Following the lines of \cite[\S\S 5, 9]{BC} and \cite[\S 6]{Ca},
we obtain some qualitative properties of the disintegrated measures
from the measure contraction property \eqref{eq:MCP}.

Given a compact set $Z \subset M$ and $\alpha \in \R$,
let us consider the set
\begin{equation}\label{eq:Z^a}
\widehat{Z}^{\alpha}
 :=\{(x,y) \in \Gamma_{\varphi} \,|\, x \in Z,\, \varphi(y)=\alpha \}.
\end{equation}
Note that, for any $(x,y) \in \widehat{Z}^{\alpha}$, we have
\begin{equation}\label{eq:Zcpt}
d(x,y)=\varphi(y)-\varphi(x)=\alpha -\varphi(x) \le \alpha -\inf_Z \varphi.
\end{equation}
Thus $\widehat{Z}^{\alpha}$ is a (possibly empty) compact set
in the closed set $\Gamma_{\varphi} \subset M \times M$.
For $\lambda \in [0,1]$, define
\begin{equation}\label{eq:Z^a_t}
Z^{\alpha}_{\lambda} :=\big\{\eta(\lambda) \,\big|\, \eta\
 \text{runs over minimal geodesics with}\ \big( \eta(0),\eta(1) \big) \in \widehat{Z}^{\alpha} \big\}.
\end{equation}
Observe that $Z^{\alpha}_0=p_1(\widehat{Z}^{\alpha}) \subset Z$ and
$Z^{\alpha}_1=p_2(\widehat{Z}^{\alpha}) \subset \varphi^{-1}(\alpha)$.
For each $\lambda \in [0,1]$, $Z^{\alpha}_{\lambda}$ is closed and hence compact
(indeed, if $\{\eta_i(\lambda)\}_{i \in \N} \subset Z^{\alpha}_{\lambda}$ is convergent,
then a subsequence $\{\eta_{i_j}\}_{j \in \N}$ is convergent in the uniform topology,
and hence $\lim_{i \to \infty}\eta_i(\lambda) \in Z^{\alpha}_{\lambda}$).

By virtue of \eqref{eq:Zcpt}, taking small $Z$ and appropriately choosing $\alpha$
allows us to assume that $\bigcup_{\lambda \in [0,1]} Z^{\alpha}_{\lambda}$
is contained in an open set $U \subset M$ on which $\Ric_N \ge K$ holds
for some $N \in (n,\infty)$ and $K \in \R$.
Then we have the following useful estimate.
Recall \eqref{eq:bs} for the definition of the function $\bs_{K/(N-1)}$.

\begin{lemma}\label{lm:MCPbd}
Let $Z \subset M$ be compact, $\alpha \in \R$, and consider
$\widehat{Z}^{\alpha}, Z^{\alpha}_{\lambda}$ as in \eqref{eq:Z^a}, \eqref{eq:Z^a_t}.
Suppose that $\bigcup_{\lambda \in [0,1]} Z^{\alpha}_{\lambda}$
is contained in an open set $U \subset M$ on which $\Ric_N \ge K$ holds
for some $N \in (n,\infty)$ and $K \in \R$.
Then we have, for all $\lambda \in (0,1)$,
\begin{equation}\label{eq:MCPbd}
\fm(Z^{\alpha}_{\lambda}) \ge (1-\lambda) \inf_{(x,y) \in \widehat{Z}^{\alpha}}
 \bigg\{ \frac{\bs_{K/(N-1)}((1-\lambda)d(x,y))}{\bs_{K/(N-1)}(d(x,y))} \bigg\}^{N-1} \fm(Z^{\alpha}_0).
\end{equation}
\end{lemma}

\begin{proof}
There is nothing to prove if $\fm(Z^{\alpha}_0)=0$,
thus we assume $\fm(Z^{\alpha}_0)>0$.
Choose a dense set $\{y_i\}_{i \in \N}$ in $Z^{\alpha}_1$ and,
for each $k \in \N$, decompose $Z^{\alpha}_0$ into
\[ A^k_i :=\{x \in Z^{\alpha}_0 \,|\, d(x,y_i)=\min_{1 \le j \le k}d(x,y_j)\},
 \qquad i=1,\ldots,k. \]
Then $A^k_i$ is a compact set and, for each $\lambda \in (0,1)$,
the measure contraction property \eqref{eq:MCP} under $\Ric_N \ge K$
(with respect to $\rev{F}$, to be precise) yields
\[ \fm\big( A^k_i(\lambda) \big) \ge (1-\lambda) \inf_{x \in A^k_i}
 \bigg\{ \frac{\bs_{K/(N-1)}((1-\lambda)d(x,y_i))}{\bs_{K/(N-1)}(d(x,y_i))} \bigg\}^{N-1} \fm(A^k_i), \]
where we set
\[ A^k_i(\lambda):=\{\eta(\lambda) \,|\, \eta\
 \text{runs over minimal geodesics with}\ \eta(0) \in A^k_i,\, \eta(1)=y_i \}. \]
We remark that, for $i \neq j$,
\[ A^k_i(\lambda) \cap A^k_j(\lambda) \subset \{ x \in M \,|\, d(x,y_i)=d(x,y_j) \} \]
and hence it has null measure.
Letting $k \to \infty$, $\bigcup_{i=1}^k A^k_i(\lambda)$
converges to $Z^{\alpha}_{\lambda}$ in the Hausdorff distance
(with respect to any distance structure comparable to the Finsler structure).
Therefore the upper semi-continuity of the measure $\fm$ with respect to
the Hausdorff distance on compact sets
(the author could not find a good reference, but see for example \cite{BV})
shows \eqref{eq:MCPbd}.
$\qedd$
\end{proof}

The bound \eqref{eq:MCPbd} is a weak one coming only from the measure contraction property,
in fact as $\lambda \to 1$ it converges to the trivial bound $\fm(Z^{\alpha}_1) \ge 0$.
We will give a sharper estimate based on the curvature-dimension condition
in the next subsection.
Nonetheless, \eqref{eq:MCPbd} can be used to obtain the important qualitative property
of disintegrated measures, namely the absolute continuity with respect to the Lebesgue measure
on $\fv$-almost all transport rays.

Recall the disintegration $\fm|_{\bm{T}_{\varphi}}=\bm{\mu}_{\gamma} \,\fv(d\gamma)$
in \eqref{eq:disint} and consider the Lebesgue decomposition
\begin{equation}\label{eq:Lebd}
\fm|_{\bm{T}_{\varphi}}
 =\rho \cdot \cG_{\sharp}[(\fv \times \fL^1)|_{\Dom(\cG)}]+\omega
\end{equation}
into the absolutely continuous and singular parts with respect to
$\cG_{\sharp}[(\fv \times \fL^1)|_{\Dom(\cG)}]$.
We remark that $\fv$ is regarded as a measure on $\bm{R}_{\varphi}$
in the former expression $\fm|_{\bm{T}_{\varphi}}=\bm{\mu}_{\gamma} \,\fv(d\gamma)$
and as a measure on $\sigma(\bm{T}_{\varphi})$ in the latter \eqref{eq:Lebd}
by the identification between $\bm{R}_{\varphi}$ and $\sigma(\bm{T}_{\varphi})$
(see \S \ref{ssc:disint}).
Recalling that $\bm{\mu}_{\gamma}$ is regarded as a measure on $\Dom(\gamma)$,
we also have the decomposition along $\gamma \in \bm{R}_{\varphi}$,
\begin{equation}\label{eq:Lebd'}
\bm{\mu}_{\gamma}
 =\rho_{\gamma} \cdot \fL^1|_{\Dom(\gamma)} +\omega_{\gamma},
\end{equation}
whose integration with respect to $\fv$ gives rise to \eqref{eq:Lebd}.

\begin{proposition}\label{pr:ac}
For $\fv$-almost every $\gamma \in \bm{R}_{\varphi}$, we have the following.
\begin{enumerate}[{\rm (i)}]
\item
$\bm{\mu}_{\gamma}$ is absolutely continuous with respect to
$\fL^1|_{\Dom(\gamma)}$, that is, $\omega_{\gamma}(\Dom(\gamma))=0$.

\item
The density function $\rho_{\gamma}:\Dom(\gamma) \lra [0,\infty)$
as in \eqref{eq:Lebd'} satisfies
\begin{equation}\label{eq:MCPbd'}
\bigg\{ \frac{\bs_{K/(N-1)}(b-t)}{\bs_{K/(N-1)}(b-s)} \bigg\}^{N-1}
 \le \frac{\rho_{\gamma}(t)}{\rho_{\gamma}(s)}
 \le \bigg\{ \frac{\bs_{K/(N-1)}(t-a)}{\bs_{K/(N-1)}(s-a)} \bigg\}^{N-1}
\end{equation}
for any $a<s<t<b$ such that $a,b \in \Dom(\gamma)$,
by assuming that $\Ric_N \ge K$ holds on $\gamma([a,b])$
with $N \in (n,\infty)$ and $K \in \R$.

\item
$\rho_{\gamma}$ is positive and locally Lipschitz on the interior of $\Dom(\gamma)$.
\end{enumerate}
\end{proposition}

\begin{proof}
(i)
Take a Borel set $A \subset \bm{T}_{\varphi}$ such that
$\cG_{\sharp}[(\fv \times \fL^1)|_{\Dom(\cG)}](A)=0$ as well as $\omega(A)=\omega(\bm{T}_{\varphi})$.
It suffices to show $\fm(A)=0$ since $\fm(A)=\omega(A)$ and
\[ \omega(A)
 =\int_{\bm{R}_{\varphi}} \omega_{\gamma}\big( \gamma^{-1}(A) \big) \,\fv(d\gamma). \]
Thus we suppose $\fm(A)>0$ and derive a contradiction.

Since $\fm$ is Borel regular,
we can take a compact set $Z \subset A$ still enjoying $\fm(Z)>0$.
For each $x \in Z$, there is $y \neq x$ with either $(x,y) \in \Gamma_{\varphi}$
or $(y,x) \in \Gamma_{\varphi}$.
By taking smaller $Z$ if necessary, we can assume that the former holds for all $x \in Z$
(consider $-\varphi$ with respect to $\rev{F}$
if $(y,x) \in \Gamma_{\varphi}$ for almost all $x \in Z$).
Then, along the unique minimal geodesic $\eta:[0,1] \lra \bm{T}_{\varphi}$
from $x$ to $y$, we have
\[ x \in Z^{\varphi(\eta(\lambda))}_0 \qquad \text{for all}\ \lambda \in [0,1]. \]
Hence for some $\alpha>\varphi(x)$ we have $\fm(Z^{\alpha}_0)>0$,
and $Z^{\alpha}_{\lambda}$ satisfies the hypothesis of Lemma~\ref{lm:MCPbd}
for some $N$ and $K$.
Then Lemma~\ref{lm:MCPbd} ensures $\fm(Z^{\alpha}_{\lambda})>0$ for all $\lambda \in [0,1)$,
and it follows from Fubini's theorem that
\begin{align*}
0 &< \int_0^{1/2} \fm(Z^{\alpha}_{\lambda}) \,d\lambda
 =\int_0^{1/2} [(\cG^{-1})_{\sharp}(\fm|_{\bm{T}_{\varphi}})]
 \big(\cG^{-1}(Z^{\alpha}_{\lambda}) \big) \,d\lambda \\
&= \int_{\Dom(\cG)} \fL^1\bigg( \bigg\{ \lambda \in \bigg[ 0,\frac{1}{2} \bigg] \,\bigg|\,
 \cG(x,s) \in Z^{\alpha}_{\lambda} \bigg\} \bigg) \,[(\cG^{-1})_{\sharp}(\fm|_{\bm{T}_{\varphi}})] (dxds).
\end{align*}
Recall that $\cG(x,s) \in Z^{\alpha}_{\lambda}$ if and only if $\cG(x,s^x_0) \in Z^{\alpha}_0$
and $\cG(x,s^x_1) \in \varphi^{-1}(\alpha)$ for some $s^x_0<s<s^x_1$ with
$s=(1-\lambda)s^x_0 +\lambda s^x_1$, and $s^x_1$ is uniquely determined by $x$.
Therefore we obtain
\begin{equation}\label{eq:ac}
0<\int_{\Dom(\cG)} \fL^1\bigg( \bigg\{ \lambda \in \bigg[ 0,\frac{1}{2} \bigg] \,\bigg|\,
 \cG\bigg( x,\frac{s-\lambda s^x_1}{1-\lambda} \bigg) \in Z^{\alpha}_0 \bigg\} \bigg)
 \,[(\cG^{-1})_{\sharp}(\fm|_{\bm{T}_{\varphi}})] (dxds).
\end{equation}

However, the hypothesis $\cG_{\sharp}[(\fv \times \fL^1)|_{\Dom(\cG)}](A)=0$
and $Z^{\alpha}_0 \subset Z \subset A$ imply
\begin{align*}
0 &=\cG_{\sharp}[(\fv \times \fL^1)|_{\Dom(\cG)}](Z^{\alpha}_0)
 =(\fv \times \fL^1)\big( \cG^{-1}(Z^{\alpha}_0) \big) \\
&=\int_{\sigma(\bm{T}_{\varphi})} \fL^1(\{t \in \R \,|\, \cG(x,t) \in Z^{\alpha}_0\}) \,\fv(dx)
\end{align*}
and hence $\fL^1(\{t \in \R \,|\, \cG(x,t) \in Z^{\alpha}_0\})=0$
for $\fv$-almost all $x \in \sigma(\bm{T}_{\varphi})$.
This contradicts \eqref{eq:ac} by noticing
\[ (p_1)_{\sharp} [(\cG^{-1})_{\sharp} (\fm|_{\bm{T}_{\varphi}})]
 =(p_1 \circ \cG^{-1})_{\sharp}(\fm|_{\bm{T}_{\varphi}})
 =\sigma_{\sharp}(\fm|_{\bm{T}_{\varphi}}) =\fv. \]

(ii)
Thanks to Lemma~\ref{lm:Rconv} and the hypothesis $\Ric_N \ge K$ on $\gamma([a,b])$,
we can apply \eqref{eq:MCPbd} with $\alpha=b$ and $\lambda=(t-s)/(b-s)$.
Put
\[ t_{\delta} :=s+\delta+\lambda\big( b-(s+\delta) \big) =t+(1-\lambda)\delta,
 \qquad \delta \in [0,b-s). \]
Then we obtain from \eqref{eq:MCPbd} and
$\fm|_{\bm{T}_{\varphi}}=(\rho_{\gamma} \cdot \fL^1|_{\Dom(\gamma)}) \,\fv(d\gamma)$ that,
for small $\ve>0$,
\[ (1-\lambda) \int_0^{\ve} \rho_{\gamma}(t_{\delta}) \,d\delta
 \ge (1-\lambda) \int_0^{\ve}
 \bigg\{ \frac{\bs_{K/(N-1)}((1-\lambda)(b-s-\delta))}{\bs_{K/(N-1)}(b-s-\delta)} \bigg\}^{N-1}
 \rho_{\gamma}(s+\delta) \,d\delta. \]
Thus we have for $\fL^1$-almost all $t$ and $s$
\begin{align*}
\rho_{\gamma}(t)
&= \lim_{\ve \downarrow 0} \frac{1}{\ve} \int_0^{\ve}
 \rho_{\gamma}(t_{\delta}) \,d\delta \\
&\ge \bigg\{ \frac{\bs_{K/(N-1)}((1-\lambda)(b-s))}{\bs_{K/(N-1)}(b-s)} \bigg\}^{N-1}
 \lim_{\ve \downarrow 0} \frac{1}{\ve} \int_0^{\ve} \rho_{\gamma}(s+\delta) \,d\delta \\
&= \bigg\{ \frac{\bs_{K/(N-1)}(b-t)}{\bs_{K/(N-1)}(b-s)} \bigg\}^{N-1} \rho_{\gamma}(s).
\end{align*}
The other inequality in \eqref{eq:MCPbd'} is derived from the same argument
for $-\varphi$ with respect to $\overleftarrow{F}$.
These estimates allow us to take the continuous version of $\rho_{\gamma}$ (see (iii) below),
and hence \eqref{eq:MCPbd'} indeed holds for all $a<s<t<b$.

(iii)
The positivity is obvious from \eqref{eq:MCPbd'},
and taking the logarithm of \eqref{eq:MCPbd'} yields the local Lipschitz continuity.
$\qedd$
\end{proof}

We remark that Klartag \cite[Theorem~1.2]{Kl} further showed
the smoothness of the density function $\rho_{\gamma}$,
while the local Lipschitz continuity is sufficient for our purpose.

\subsection{Quantitative estimates derived from $\CD$}\label{ssc:quan}%%%%%%%%%%%%%
%%%%%%%%%%%%%%%%

Having the properties of $\rho_{\gamma}$ in Proposition~\ref{pr:ac} at hand,
we derive the following sharper quantitative estimate from the curvature-dimension condition
along the lines of \cite[Theorem~4.2]{CM}.
Define
\[ \bm{\sigma}^{(\lambda)}_{\kappa}(r)
 :=\frac{\bs_{\kappa}(\lambda r)}{\bs_{\kappa}(r)},
 \qquad \kappa \in \R,\, \lambda \in (0,1), \]
for $r>0$ if $\kappa \le 0$ and for $r \in (0,\pi/\sqrt{\kappa})$ if $\kappa>0$.
We set $\bm{\sigma}^{(\lambda)}_{\kappa}(0):=\lambda$ and, when $\kappa>0$,
$\bm{\sigma}^{(\lambda)}_{\kappa}(r):=\infty$ for $r \ge \pi/\sqrt{\kappa}$.
Recall from \eqref{eq:btau} that
$\bm{\tau}^{(\lambda)}_{K,N}(r)^N =\lambda \bm{\sigma}^{(\lambda)}_{K/(N-1)}(r)^{N-1}$
for $N \neq 0$.

\begin{theorem}[$N \in (-\infty,0) \cup [n,\infty)$ case]\label{th:CDbd}
Suppose that $(M,F,\fm)$ satisfies $\Ric_N \ge K$
for some $N \in (-\infty,0) \cup [n,\infty)$ $($with $N \neq 1$ if $n=1)$
and $K \in \R$.
Then, along $\fv$-almost every $\gamma \in \bm{R}_{\varphi}$, we have
\begin{equation}\label{eq:CDbd}
\rho_{\gamma}\big( (1-\lambda)s+\lambda t \big)
 \ge \left\{ \bm{\sigma}_{K/(N-1)}^{(1-\lambda)}(t-s) \rho_{\gamma}(s)^{1/(N-1)}
 +\bm{\sigma}_{K/(N-1)}^{(\lambda)}(t-s) \rho_{\gamma}(t)^{1/(N-1)} \right\}^{N-1}
\end{equation}
for all $a<s<t<b$ with $a,b \in \Dom(\gamma)$ and $\lambda \in (0,1)$.
\end{theorem}

\begin{proof}
One can prove \eqref{eq:CDbd} by a localization argument similar to
the derivation of \eqref{eq:MCPbd'} from \eqref{eq:MCPbd} with the help of Lemma~\ref{lm:Rconv}.
Fix a compact set $Z \subset M$, $\alpha \in \R$ and $l \gg \ve>0$ satisfying
$\fv(\sigma(Z^{\alpha}_1))>0$ and $\sup_{Z^{\alpha}_0} \varphi \le \alpha -l-\ve$,
where $Z^{\alpha}_0$ is as in \eqref{eq:Z^a_t}.
Put $s:=\alpha-l$, $t:=\alpha$, $t_{\lambda}:=(1-\lambda)s+\lambda t$ and
\[ \mu_{\lambda}:=\frac{1}{\fv(\sigma(Z^{\alpha}_1)) \cdot c_{\lambda} \ve}
 \cdot \cG_{\sharp} \big[ (\fv \times \fL^1)|_{\Omega_{\lambda}} \big] \in \cP(M) \]
for $\lambda \in [0,1]$, where
\begin{align*}
\Omega_{\lambda}
&:=\{ (x,r) \,|\, x \in \sigma(Z^{\alpha}_1),\,
 r \in [t_{\lambda}-c_{\lambda} \ve-\varphi(x),t_{\lambda}-\varphi(x)] \}
 \subset \Dom(\cG), \\
c_{\lambda} &:=(1-\lambda)+\lambda c,
\end{align*}
and $c>0$ is a constant (depending on $l$) chosen later in \eqref{eq:c}.
Notice that, since $\ve \ll l$,
\[ \alpha-l-\ve \le t_{\lambda}-c_{\lambda}\ve \le r+\varphi(x) \le t_{\lambda} \le \alpha \]
for $(x,r) \in \Omega_{\lambda}$,
thus indeed $\Omega_{\lambda} \subset \Dom(\cG)$.
It follows from Lemma~\ref{lm:d^2} that the set
\[ \left\{ \big( \cG(x,s-\delta-\varphi(x)),\cG(x,t-c\delta-\varphi(x)) \big) \,\big|\,
 x \in \sigma(Z^{\alpha}_1),\, \delta \in [0,\ve] \right\} \subset \Gamma_{\varphi} \]
is $d^2$-cyclically monotone.
Therefore $(\mu_{\lambda})_{\lambda \in [0,1]}$ is a minimal geodesic with respect to $W_2$.
Moreover, by construction,
$\mu_{\lambda}=(T_{\lambda})_{\sharp}\mu_0$ holds with
\[ T_{\lambda}\big( \cG(x,s-\delta-\varphi(x)) \big)
 :=\cG\big( x,t_{\lambda}-c_{\lambda} \delta-\varphi(x) \big) \qquad
 \text{for}\ x \in \sigma(Z^{\alpha}_1),\, \delta \in [0,\ve]. \]
Let us rewrite $\mu_{\lambda}$ by using \eqref{eq:Lebd} as
\[ \mu_{\lambda} =\frac{\rho^{-1}}{\fv(\sigma(Z^{\alpha}_1)) \cdot c_{\lambda} \ve} \cdot
 \fm|_{\cG(\Omega_{\lambda})}. \]
Applying the curvature-dimension condition \eqref{eq:CD} or \eqref{eq:CDneg}
following from $\Ric_N \ge K$
to $(\mu_{\lambda})_{\lambda \in [0,1]}$, we obtain
\begin{align*}
&\bigg( \int_{\supp\mu_{\lambda}} (c_{\lambda}^{-1} \rho^{-1})^{(N-1)/N} \,d\fm \bigg)^N \\
&\ge \bigg( \int_{\supp \mu_0} \bm{\tau}_{K,N}^{(1-\lambda)}\big( d(z,T_1(z)) \big)
 \rho(z)^{(1-N)/N} \,\fm(dz) \\
&\qquad +\int_{\supp \mu_1} \bm{\tau}_{K,N}^{(\lambda)}\big( d(T_1^{-1}(z),z) \big)
 \big( c\rho(z) \big)^{(1-N)/N} \,\fm(dz) \bigg)^N \\
&= \bigg( \int_0^{\ve} \int_{\sigma(Z^{\alpha}_1)}
 \bm{\tau}_{K,N}^{(1-\lambda)}\big( l+(1-c)\delta \big)
 \rho\big( \cG(x,s-\delta-\varphi(x)) \big)^{1/N} \,\fv(dx) d\delta \\
&\qquad +\int_0^{c\ve} \int_{\sigma(Z^{\alpha}_1)}
 \bm{\tau}_{K,N}^{(\lambda)}\big( l+(c^{-1}-1)\delta \big)
 c^{(1-N)/N} \rho\big( \cG(x,t-\delta-\varphi(x)) \big)^{1/N} \,\fv(dx) d\delta \bigg)^N.
\end{align*}
The localization of this inequality (as $\delta \downarrow 0$) leads to,
for $\fv$-almost all $x \in \sigma(Z^{\alpha}_1)$,
\begin{align}
&c_{\lambda} \rho_{\gamma}\big( \cG(x,t_{\lambda}-\varphi(x)) \big) \nonumber\\
&\ge \left\{ \bm{\tau}_{K,N}^{(1-\lambda)}(l) \rho_{\gamma}\big( \cG(x,s-\varphi(x)) \big)^{1/N}
 +\bm{\tau}_{K,N}^{(\lambda)}(l) c^{1/N} \rho_{\gamma}\big( \cG(x,t-\varphi(x)) \big)^{1/N} \right\}^N.
 \label{eq:CDwbd}
\end{align}

Set $f(\lambda):=\rho_{\gamma}(\cG(x,t_{\lambda}-\varphi(x)))$ and
let us rewrite \eqref{eq:CDwbd} as
\[ f(\lambda) \ge
 \left\{ \bm{\tau}_{K,N}^{(1-\lambda)}(l) c_{\lambda}^{-1/N} f(0)^{1/N}
 +\bm{\tau}_{K,N}^{(\lambda)}(l) (c_{\lambda}/c)^{-1/N} f(1)^{1/N} \right\}^N. \]
By calculation, the RHS is maximized when
\begin{equation}\label{eq:c}
c=\left\{ \frac{(1-\lambda) \bm{\tau}^{(\lambda)}_{K,N}(l) f(1)^{1/N}}
 {\lambda \bm{\tau}^{(1-\lambda)}_{K,N}(l) f(0)^{1/N}} \right\}^{N/(N-1)}.
\end{equation}
On the one hand, substituting this into \eqref{eq:CDwbd} yields
\begin{align*}
c_{\lambda} f(\lambda)
&\ge \left\{ \bm{\sigma}^{(1-\lambda)}_{K/(N-1)}(l) f(0)^{1/(N-1)}
 +\bm{\sigma}^{(\lambda)}_{K/(N-1)}(l) f(1)^{1/(N-1)} \right\}^N \\
&\qquad \times \left\{ \frac{1-\lambda}{\bm{\tau}^{(1-\lambda)}_{K,N}(l) f(0)^{1/N}}
 \right\}^{N/(N-1)}.
\end{align*}
On the other hand, we find
\[ c_{\lambda}
 = \left\{ \bm{\sigma}^{(1-\lambda)}_{K/(N-1)}(l) f(0)^{1/(N-1)}
 +\bm{\sigma}^{(\lambda)}_{K/(N-1)}(l) f(1)^{1/(N-1)} \right\}
 \left\{ \frac{1-\lambda}{\bm{\tau}^{(1-\lambda)}_{K,N}(l) f(0)^{1/N}}
 \right\}^{N/(N-1)}. \]
Therefore we obtain
\[ f(\lambda)
 \ge \left\{ \bm{\sigma}^{(1-\lambda)}_{K/(N-1)}(l) f(0)^{1/(N-1)}
 +\bm{\sigma}^{(\lambda)}_{K/(N-1)}(l) f(1)^{1/(N-1)} \right\}^{N-1}. \]
This is exactly the desired inequality \eqref{eq:CDbd}
(recall that the domain $\Dom(\gamma)$ was taken so as to satisfy $\varphi(\gamma(r))=r$).
$\qedd$
\end{proof}

\begin{remark}\label{rm:CDbd}
(a)
Notice that choosing $c=1$ in \eqref{eq:CDwbd} gives only the weaker bound
\[ \rho_{\gamma}\big( (1-\lambda)s+\lambda t \big)
 \ge \left\{ \bm{\tau}_{K,N}^{(1-\lambda)}(t-s) \rho_{\gamma}(s)^{1/N}
 +\bm{\tau}_{K,N}^{(\lambda)}(t-s) \rho_{\gamma}(t)^{1/N} \right\}^N. \]
The above improving argument employing a shrinking ($c<1$) or expanding ($c>1$) transport
would be compared with a similar technique in the implication from $\CD(K,N)$ to $\Ric_N \ge K$
via the Brunn--Minkowski inequality
(recall the role of $a$ in the proof of Theorem~\ref{th:CD}).

(b)
If $\rho_{\gamma} \in \cC^2(\Dom(\gamma))$, then comparing \eqref{eq:CDbd} with
\begin{align*}
&\frac{\del^2}{\del \lambda^2} \left[ \bm{\sigma}_{K/(N-1)}^{(1-\lambda)}(t-s) \rho_{\gamma}(s)^{1/(N-1)}
 +\bm{\sigma}_{K/(N-1)}^{(\lambda)}(t-s) \rho_{\gamma}(t)^{1/(N-1)} \right] \\
&= -\frac{K(t-s)^2}{N-1} \left[ \bm{\sigma}_{K/(N-1)}^{(1-\lambda)}(t-s) \rho_{\gamma}(s)^{1/(N-1)}
 +\bm{\sigma}_{K/(N-1)}^{(\lambda)}(t-s) \rho_{\gamma}(t)^{1/(N-1)} \right]
\end{align*}
shows that
\begin{align*}
\left( \rho_{\gamma}^{1/(N-1)} \right)'' +\frac{K}{N-1} \rho_{\gamma}^{1/(N-1)}
&\le 0 \qquad \text{for}\ N \in [n,\infty), \\
\left( \rho_{\gamma}^{1/(N-1)} \right)'' +\frac{K}{N-1} \rho_{\gamma}^{1/(N-1)}
&\ge 0 \qquad \text{for}\ N \in (-\infty,0).
\end{align*}
This is equivalent to, in either case,
\[ \psi'' -\frac{(\psi')^2}{N-1} \ge K, \qquad
 \text{where}\ \rho_{\gamma}=\e^{-\psi}. \]
Hence $(\Dom(\gamma),|\cdot|,\bm{\mu}_{\gamma})$ satisfies $\Ric_N \ge K$.
We stress that the standard distance $|\cdot|$ cannot be replaced with
the distance function induced from the Finsler structure $F$,
since we do not know anything about the behavior of the reverse curve of $\gamma$
(that is not necessarily a geodesic nor of constant speed with respect to $F$).
\end{remark}

\begin{theorem}[$N=0$ case]
Suppose that $(M,F,\fm)$ satisfies $\Ric_0 \ge K$ for some $K \in \R$.
Then, along $\fv$-almost every $\gamma \in \bm{R}_{\varphi}$, we have
\begin{equation}\label{eq:CD0bd}
\rho_{\gamma}\big( (1-\lambda)s+\lambda t \big)
 \ge \Big\{ \bm{\sigma}_{-K}^{(1-\lambda)}(t-s) \rho_{\gamma}(s)^{-1}
 +\bm{\sigma}_{-K}^{(\lambda)}(t-s) \rho_{\gamma}(t)^{-1} \Big\}^{-1}
\end{equation}
for all $a<s<t<b$ with $a,b \in \Dom(\gamma)$ and $\lambda \in (0,1)$.
\end{theorem}

\begin{proof}
The proof is a modification of that of Theorem~\ref{th:CDbd}.
Instead of \eqref{eq:CDwbd}, we obtain
\[ c_{\lambda}^{-1} f(\lambda)^{-1}
 \le \max\bigg\{ \frac{\bm{\sigma}_{-K}^{(1-\lambda)}(l)}{1-\lambda} f(0)^{-1},
 \frac{\bm{\sigma}_{-K}^{(\lambda)}(l)}{c\lambda} f(1)^{-1} \bigg\}, \]
where $f(\lambda):=\rho_{\gamma}(\cG(x,t_{\lambda}-\varphi(x)))$ as before.
Choosing
\[ c=\frac{(1-\lambda) \bm{\sigma}_{-K}^{(\lambda)}(l) f(0)}
 {\lambda \bm{\sigma}_{-K}^{(1-\lambda)}(l) f(1)}, \]
we have
\[ c_{\lambda}^{-1} f(\lambda)^{-1} \le
 \frac{\bm{\sigma}_{-K}^{(1-\lambda)}(l)}{1-\lambda} f(0)^{-1}. \]
Note also that
\[ c_{\lambda} =\big\{ \bm{\sigma}_{-K}^{(1-\lambda)}(l)f(0)^{-1}
 +\bm{\sigma}_{-K}^{(\lambda)}(l)f(1)^{-1} \big\}
 \cdot \frac{(1-\lambda)f(0)}{\bm{\sigma}_{-K}^{(1-\lambda)}(l)}. \]
Thus we obtain
\[ f(\lambda) \ge
 c_{\lambda}^{-1} \frac{(1-\lambda)f(0)}{\bm{\sigma}_{-K}^{(1-\lambda)}(l)}
 =\big\{ \bm{\sigma}_{-K}^{(1-\lambda)}(l)f(0)^{-1}
 +\bm{\sigma}_{-K}^{(\lambda)}(l)f(1)^{-1} \big\}^{-1}. \]
$\qedd$
\end{proof}

We remark that the estimate \eqref{eq:CD0bd}
is indeed the limit of \eqref{eq:CDbd} as $N \uparrow 0$.
The case of $N=\infty$ is simpler and goes as follows.

\begin{theorem}[$N=\infty$ case]\label{th:CDbd'}
Suppose that $(M,F,\fm)$ satisfies $\Ric_{\infty} \ge K$ for some $K \in \R$.
Then, along $\fv$-almost every $\gamma \in \bm{R}_{\varphi}$, we have
\begin{equation}\label{eq:CDbd'}
 \log\rho_{\gamma}\big( (1-\lambda)s+\lambda t \big)
 \ge (1-\lambda) \log\rho_{\gamma}(s) +\lambda \log\rho_{\gamma}(t)
 +\frac{K}{2}(1-\lambda)\lambda (t-s)^2
\end{equation}
for all $a<s<t<b$ with $a,b \in \Dom(\gamma)$ and $\lambda \in (0,1)$.
\end{theorem}

\begin{proof}
Arguing as in the proof of Theorem~\ref{th:CDbd} with $c=1$,
we have by $\Ric_{\infty} \ge K$
\begin{align*}
&\int_{\supp\mu_{\lambda}} \rho^{-1} \log(\rho^{-1}) \,d\fm \\
&\le (1-\lambda) \int_0^{\ve} \int_{\sigma(Z^{\alpha}_1)}
 \log\big[ \rho\big( \cG(x,s-\delta-\varphi(x)) \big)^{-1} \big] \,\fv(dx) d\delta \\
&\quad +\lambda \int_0^{\ve} \int_{\sigma(Z^{\alpha}_1)}
 \log\big[ \rho\big( \cG(x,t-\delta-\varphi(x)) \big)^{-1} \big] \,\fv(dx) d\delta \\
&\quad -\fv\big( \sigma(Z^{\alpha}_1) \big) \ve \cdot
 \frac{K}{2}(1-\lambda)\lambda W_2^2(\mu_0,\mu_1).
\end{align*}
The localization gives
\begin{align*}
&\log\rho_{\gamma}\big( \cG(x,t_{\lambda}-\varphi(x)) \big) \\
&\ge (1-\lambda) \log\rho_{\gamma}\big( \cG(x,s-\varphi(x)) \big)
 +\lambda \log\rho_{\gamma}\big( \cG(x,t-\varphi(x)) \big)
 +\frac{K}{2}(1-\lambda)\lambda l^2
\end{align*}
which shows \eqref{eq:CDbd'}.
$\qedd$
\end{proof}

\begin{remark}\label{rm:CDbd'}
Similarly to Remark~\ref{rm:CDbd}(b), if $\rho_{\gamma} \in \cC^2(\Dom(\gamma))$,
then the inequalities \eqref{eq:CD0bd} and \eqref{eq:CDbd'} imply
$\psi''+(\psi')^2 \ge K$ and $\psi'' \ge K$, respectively.
\end{remark}

One can deduce \eqref{eq:CDbd'} also from \eqref{eq:CDbd} as $N \downarrow -\infty$.
Indeed, $\Ric_N \ge \Ric_{\infty} \ge K$ for $N<0$ and, by putting $\ve=-1/(N-1)$,
\begin{align*}
&(N-1) \log\left[ \bm{\sigma}^{(1-\lambda)}_{K/(N-1)}(l) a^{1/(N-1)}
 +\bm{\sigma}^{(\lambda)}_{K/(N-1)}(l) b^{1/(N-1)} \right] \\
&= -\frac{1}{\ve} \log \left[ \frac{\bs_{-K\ve}((1-\lambda)l)}{\bs_{-K\ve}(l)} a^{-\ve}
 +\frac{\bs_{-K\ve}(\lambda l)}{\bs_{-K\ve}(l)} b^{-\ve} \right] \\
&\to (1-\lambda) \log a +\lambda \log b +\frac{K}{2}(1-\lambda)\lambda l^2
\end{align*}
as $\ve \downarrow 0$ for $a,b>0$.

\section{Isoperimetric inequalities}\label{sc:isop}%%%%%%%%%%%%%%%%%%%%%%
%%%%%%%%%%%%%%%%

We prove in this section the isoperimetric inequality (Theorem~\ref{th:isop}).
We follow the argument in \cite[\S 6]{CM}, however,
the non-reversibility finally makes a difference.

\subsection{Proof of Theorem~\ref{th:isop}}\label{ssc:isop}%%%%%%%%%%%
%%%%%%%%%%%%%%%%%%%%%%

The strategy is to reduce the isoperimetric inequality \eqref{eq:Fisop}
to the one-dimensional isoperimetric inequalities on needles
$(\Dom(\gamma),F,\bm{\mu}_{\gamma})$.
Then, since the density function $\rho_{\gamma}$ of $\bm{\mu}_{\gamma}$
is only locally Lipschitz, we approximate $\rho_{\gamma}$ by smooth functions
and apply the isoperimetric inequality in \cite{Misharp,Mineg}.
Before discussing the general situation, let us consider the special case of $n=N=1$.

\begin{step}[The case of $n=N=1$]
%%%%%%%%%%%%%%%%%%%%
In this extremal case, only $K=0$ is meaningful and $D<\infty$ is necessary for $\fm(M)=1$
(recall Remark~\ref{rm:wRic}(c)).
Thus it is sufficient to consider the case where
$(M,F)$ is diffeomorphic to the circle $\Sph^1=\R/\Z$ with
$F(\del/\del x) \equiv D$ and $F({-}\del/\del x) \equiv \Lambda_{(M,F)}^{-1} \cdot D$,
and $\fm$ coincides with the standard (Hausdorff) measure $dx$.
Therefore we have, for all $\theta \in (0,1)$,
\[ \cI_{(M,F,\fm)}(\theta) =\frac{1}{D}+\frac{1}{\Lambda_{(M,F)}^{-1} \cdot D}
 =\frac{1+\Lambda_{(M,F)}}{D}. \]
This bound is better than $\cI_{0,1,D}(\theta)=D^{-1}$
for the reason that $\del M=\emptyset$ is assumed.
If one admits manifolds with boundary, then for $M=[0,1]$
we indeed have $\cI(\theta)=D^{-1}$.
\end{step}

\begin{step}[Reduction to needles]
%%%%%%%%%%%%%%%%%%%%
Assume $N \neq 1$ from here on.
Since $\cI_{(M,F,\fm)}(\theta)=\cI_{K,N,D}(\theta)$ clearly holds
for $\theta=0,1$ (recall \eqref{eq:I} for the definition of $\cI_{(M,F,\fm)}(\theta)$),
we fix $\theta \in (0,1)$.
Given a Borel set $A \subset M$ with $\fm(A)=\theta$,
we consider the function
\[ f(x):=\chi_A(x)-\theta, \qquad x \in M. \]
Clearly $\int_M f \,d\fm=0$ holds and hence the argument in \S \ref{sc:mean0} applies.
Let $\varphi:M \lra \R$ be a $1$-Lipschitz function given by Lemma~\ref{lm:MK}.
Since $f \neq 0$ on whole $M$,
we have $\fm(\bm{D}_{\varphi})=0$ by Proposition~\ref{pr:mean0}.
It also follows from Proposition~\ref{pr:mean0} that,
for $\fv$-almost all $\gamma \in \bm{R}_{\varphi}$,
\[ 0=\int_{\Dom(\gamma)} f \circ \gamma \,d\bm{\mu}_{\gamma}
 =\bm{\mu}_{\gamma} \big( \gamma^{-1}(A) \big) -\theta. \]
Hence we have $\bm{\mu}_{\gamma}(\gamma^{-1}(A)) =\theta$
and, by the definition of the isoperimetric profile,
\[ \bm{\mu}_{\gamma}^+ \big( \gamma^{-1}(A) \big)
 \ge \cI_{(\Dom(\gamma),F_{\gamma},\bm{\mu}_{\gamma})}(\theta), \]
where $F_{\gamma}$ denotes the Finsler structure induced from $F|_{\Image(\gamma)}$
by identifying $\Dom(\gamma)$ and $\Image(\gamma)$.
We shall see in the following steps that
\begin{equation}\label{eq:1D}
\cI_{(\Dom(\gamma),F_{\gamma},\bm{\mu}_{\gamma})}(\theta)
 \ge \Lambda_{(M,F)}^{-1} \cdot \cI_{K,N,D}(\theta).
\end{equation}
It follows from Fubini's theorem and Fatou's lemma that
\begin{align*}
\fm^+(A) &= \liminf_{\ve \downarrow 0} \frac{\fm(B^+(A,\ve))-\fm(A)}{\ve} \\
&= \liminf_{\ve \downarrow 0} \int_{\bm{R}_{\varphi}}
 \frac{\bm{\mu}_{\gamma}(\gamma^{-1}(B^+(A,\ve)))
 -\bm{\mu}_{\gamma}(\gamma^{-1}(A))}{\ve} \,\fv(d\gamma) \\
&\ge \int_{\bm{R}_{\varphi}} \liminf_{\ve \downarrow 0}
 \frac{\bm{\mu}_{\gamma}(\gamma^{-1}(B^+(A,\ve)))
 -\bm{\mu}_{\gamma}(\gamma^{-1}(A))}{\ve} \,\fv(d\gamma).
\end{align*}
Notice that $\gamma^{-1}(B^+(A,\ve)) \supset B^+(\gamma^{-1}(A),\ve)$,
where $B^+(\gamma^{-1}(A),\ve)$ is the $\ve$-neighborhood in $\Dom(\gamma)$
with respect to $F_{\gamma}$.
Assuming \eqref{eq:1D}, we obtain
\[ \liminf_{\ve \downarrow 0}
 \frac{\bm{\mu}_{\gamma}(B^+(\gamma^{-1}(A),\ve))
 -\bm{\mu}_{\gamma}(\gamma^{-1}(A))}{\ve}
 =\bm{\mu}_{\gamma}^+ \big( \gamma^{-1}(A) \big)
 \ge \Lambda_{(M,F)}^{-1} \cdot \cI_{K,N,D}(\theta). \]
Therefore we conclude that the desired isoperimetric inequality
\[ \fm^+(A) \ge \Lambda_{(M,F)}^{-1} \cdot \cI_{K,N,D}(\theta) \]
holds.
There only remains to show the one-dimensional inequality \eqref{eq:1D}.
\end{step}

\begin{step}[Reduction to smooth densities]
%%%%%%%%%%%%%%%%%%%%%
We next reduce \eqref{eq:1D} to smooth density functions similarly to \cite{CM},
in order to apply the smooth arguments in \cite{Misharp,Mineg}.
Precisely, we apply the following lemma (borrowed from \cite[Lemma~6.2]{CM})
to the density function $\rho_{\gamma}:\Dom(\gamma) \lra [0,\infty)$
of $\bm{\mu}_{\gamma}$.
As a mollifier, let us fix an arbitrary non-negative function $\phi \in \cC^{\infty}(\R)$
with $\supp\phi \subset [0,1]$ and $\int_{\R} \phi \,d\fL^1=1$.

\begin{lemma}\label{lm:smooth}
Let $I_D:=[0,D]$ for $D \in (0,\infty)$ and $I_{\infty}:=\R$,
and take a non-negative continuous function $\rho:\R \lra [0,\infty)$
with $\int_{\R} \rho \,dt=1$ and $\supp\rho \subset I_D$.
For $N \in (-\infty,0] \cup [n,\infty)$, $\ve>0$ and $t \in \R$,
put $\phi_{\ve}(t):=\phi(t/\ve)/\ve$ and consider
\[ \rho_{N,\ve}(t):=(\rho^{1/(N-1)} * \phi_{\ve})(t)^{N-1}
 =\bigg( \int_{\R} \rho(t-s)^{1/(N-1)} \phi_{\ve}(s) \,ds \bigg)^{N-1},
 \quad t \in \R, \]
and
\[ \rho_{\infty,\ve}(t) :=\exp\big[ (\log\rho * \phi_{\ve})(t) \big]
 =\exp\bigg[ \int_{\R} \log\big( \rho(t-s) \big) \phi_{\ve}(s) \,ds \bigg], \quad t \in \R. \]
Then we have the following.
\begin{enumerate}[{\rm (i)}]
\item
$\rho_{N,\ve} \in \cC^{\infty}(\R)$ with $\supp\rho_{\ve} \subset I_{D+\ve}$,
and $\rho_{N,\ve}$ converges to $\rho$ uniformly on each bounded interval as $\ve \downarrow 0$.
If $D=\infty$, we also have $\rho_{N,\ve} \to \rho$ in $L^1$.

\item
If $\rho$ satisfies \eqref{eq:CDbd}
$($or \eqref{eq:CD0bd} if $N=0$, \eqref{eq:CDbd'} if $N=\infty)$,
then $\rho_{N,\ve}$ satisfies
\begin{equation}\label{eq:psiNe}
\psi_{N,\ve}'' -\frac{(\psi_{N,\ve}')^2}{N-1} \ge K,
 \qquad \text{where}\ \rho_{N,\ve}=\e^{-\psi_{N,\ve}}.
\end{equation}
\end{enumerate}
\end{lemma}

\begin{proof}
Since the proofs are common, we consider only $N \in (-\infty,0] \cup [n,\infty)$.

(i)
It is a standard fact that $\rho_{N,\ve}^{1/(N-1)}$ converges to $\rho^{1/(N-1)}$
uniformly on each bounded interval, then the first assertion follows.
When $D=\infty$, observe from Jensen's inequality that
\[ \rho_{N,\ve}(t) \le \int_{\R} \rho(t-s) \phi_{\ve}(s) \,ds =: \rho_{\ve}(t). \]
Hence $-\rho \le \rho_{N,\ve}-\rho \le \rho_{\ve}-\rho$,
then the dominated convergence theorem and
$\rho_{\ve} \to \rho$ in $L^1$ show that $\rho_{N,\ve} \to \rho$ in $L^1$.

(ii)
We first remark that $\supp\rho$ is convex (thus a closed interval) by
\eqref{eq:CDbd} or \eqref{eq:CD0bd}.
Take $a<t_0<t_1<b$ with $a,b \in \supp\rho$ and put $l:=t_1 -t_0$.
Then we observe for $\lambda \in (0,1)$
\begin{align*}
&\rho_{N,\ve} \big( (1-\lambda) t_0+\lambda t_1 \big)
 =(\rho^{1/(N-1)} * \phi_{\ve}) \big( (1-\lambda) t_0+\lambda t_1 \big)^{N-1} \\
&= \bigg( \int_{\R} \rho \big( (1-\lambda)(t_0 -s)+\lambda (t_1 -s) \big)^{1/(N-1)}
 \phi_{\ve}(s) \,ds \bigg)^{N-1} \\
&\ge \bigg( \int_{\R} \Big\{ \bm{\sigma}^{(1-\lambda)}_{K/(N-1)}(l) \rho(t_0 -s)^{1/(N-1)}
 +\bm{\sigma}^{(\lambda)}_{K/(N-1)}(l) \rho(t_1 -s)^{1/(N-1)} \Big\} \phi_{\ve}(s) \,ds \bigg)^{N-1} \\
&= \Big\{ \bm{\sigma}^{(1-\lambda)}_{K/(N-1)}(l) \rho_{N,\ve}(t_0)^{1/(N-1)}
 +\bm{\sigma}^{(\lambda)}_{K/(N-1)}(l) \rho_{N,\ve}(t_1)^{1/(N-1)} \Big\}^{N-1}.
\end{align*}
This concavity inequality implies \eqref{eq:psiNe} as we discussed in Remark~\ref{rm:CDbd}(b)
(and Remark~\ref{rm:CDbd'} for $N=0,\infty$).
$\qedd$
\end{proof}

When we apply Lemma~\ref{lm:smooth} to $\rho=\rho_{\gamma}$,
the resulting smooth function needs to be normalized since
$\int_{\R} \rho_{N,\ve} \,dt =:m_{N,\ve}$ may not be $1$.
Notice that $\rho_{N,\ve}/m_{N,\ve}$ still enjoys \eqref{eq:psiNe}
since $\rho_{N,\ve}/m_{N,\ve}=\e^{-\psi_{N,\ve}-\log m_{N,\ve}}$.
Thus it follows from Lemma~\ref{lm:smooth}
and the isoperimetric inequality \eqref{eq:M-isop} in \cite{Misharp,Mineg} that
\[ \cI_{(\Dom(\gamma),|\cdot|,\bm{\mu}_{\gamma})}(\theta)
 \ge \lim_{\ve \downarrow 0} \cI_{K,N,D+\ve}(\theta)
 =\cI_{K,N,D}(\theta), \]
where the latter equality holds
thanks to the precise formula of $\cI_{K,N,D}(\theta)$ in \cite{Misharp,Mineg}.
(We remark that, to be precise, when $\Dom(\gamma)$ is a half-infinite interval
we use a variant of Lemma~\ref{lm:smooth} on $[0,\infty)$ or $(-\infty,0]$.)
\end{step}

\begin{step}[Proof of \eqref{eq:1D}]\label{step3}
%%%%%%%%%%%%%%%%%%%%%
We finally need to compare $F$ and $|\cdot|$ along $\gamma$,
this is the only difference between the reversible and non-reversible cases.
For $s,t \in \Dom(\gamma)$ with $s<t$, we observe by the definition of the Finsler distance
and the reversibility constant $\Lambda_{(M,F)}$ that
\[ d\big( \gamma(s),\gamma(t) \big)=t-s, \qquad
 d\big( \gamma(t),\gamma(s) \big) \le \Lambda_{(M,F)} \cdot (t-s). \]
Therefore we have \eqref{eq:1D} as
\[ \cI_{(\Dom(\gamma),F_{\gamma},\bm{\mu}_{\gamma})}(\theta)
 \ge \Lambda_{(M,F)}^{-1} \cdot \cI_{(\Dom(\gamma),|\cdot|,\bm{\mu}_{\gamma})}(\theta)
 \ge \Lambda_{(M,F)}^{-1} \cdot \cI_{K,N,D}(\theta), \]
and complete the proof of Theorem~\ref{th:isop}.
$\qedd$
\end{step}

\begin{remark}\label{rm:proof}
In Step~\ref{step3}, one can readily find an isoperimetric minimizer $I \subset \Dom(\gamma)$
satisfying $\bm{\mu}_{\gamma}(I)=\theta$ and
$\bm{\mu}_{\gamma}^+(I)=\cI_{\Dom(\gamma),F_{\gamma},\bm{\mu}_{\gamma}}(\theta)$.
If $I$ is of the form $I=[a,c]$ with $\Dom(\gamma)=[a,b]$
or $I=(-\infty,c]$ with $\Dom(\gamma)=(-\infty,b]$,
then in the calculation of $\bm{\mu}_{\gamma}^+(I)$
only $d(\gamma(s),\gamma(t))=t-s$ for $s<t$ matters,
and we obtain
$\cI_{(\Dom(\gamma),F_{\gamma},\bm{\mu}_{\gamma})}(\theta) \ge \cI_{K,N,D}(\theta)$.
This condition on $I$, however, seems not true for general $(M,F,\fm)$.
\end{remark}

\subsection{Corollaries}\label{ssc:cor}%%%%%%%%%%%%%
%%%%%%%%%%%%%%%%%%%%%%%

We state two special cases of Theorem~\ref{th:isop} as corollaries.

\begin{corollary}[L\'evy--Gromov's isoperimetric inequality]\label{cr:LG}
If $K>0$, $N \in [n,\infty)$ $($with $N \neq 1$ if $n=1)$
and $D \ge \pi\sqrt{(N-1)/K}$ in $Theorem~\ref{th:isop}$,
then we have
\[ \cI_{(M,F,\fm)}(\theta) \ge \Lambda_{(M,F)}^{-1} \cdot
 \frac{\bs_{K/(N-1)}(R(\theta))^{N-1}}{\int_0^{\pi\sqrt{(N-1)/K}} \bs_{K/(N-1)}(r)^{N-1} \,dr}, \]
where $R(\theta)$ is given by
\[ \int_0^{R(\theta)} \bs_{K/(N-1)}(r)^{N-1} \,dr
 =\theta \int_0^{\pi\sqrt{(N-1)/K}} \bs_{K/(N-1)}(r)^{N-1} \,dr. \]
\end{corollary}

\begin{corollary}[Bakry--Ledoux's isoperimetric inequality]\label{cr:BL}
If $K>0$ and $N=D=\infty$ in $Theorem~\ref{th:isop}$,
then we have
\[ \cI_{(M,F,\fm)}(\theta) \ge \Lambda_{(M,F)}^{-1} \cdot
 \sqrt{\frac{K}{2\pi}} \e^{-Ka(\theta)^2/2}, \qquad
 \text{where}\ \theta=\int_{-\infty}^{a(\theta)} \sqrt{\frac{K}{2\pi}} \e^{-Ks^2/2} \,ds. \]
\end{corollary}

Notice finally that, although our Finsler structures are necessarily $\cC^{\infty}$ and strongly convex,
bi-Lipschitz approximations give the following isoperimetric inequality on general normed spaces.

\begin{corollary}\label{cr:norm}
Let $n \ge 2$, $N \in (-\infty,0] \cup [n,\infty]$, $K \in \R$,
and $\|\cdot\|:\R^n \lra [0,\infty)$ be a continuous function satisfying$:$
\begin{enumerate}[{\rm (1)}]
\item $\|x\| >0$ for all $x \neq 0:= (0,\ldots,0);$
\item $\|cx\|=c\|x\|$ for any $x \in \R^n$ and $c>0;$
\item $\|x+y\| \le \|x\| +\|y\|$ for any $x,y \in \R^n$.
\end{enumerate}
Take a probability measure $\fm=\rho \fL^n$ on $\R^n$
absolutely continuous with respect to
the $n$-dimensional Lebesgue measure $\fL^n$ such that
$\rho$ is positive, continuous and satisfies$:$
\begin{equation}\label{eq:normCD}
\rho\big( (1-\lambda)x+\lambda y \big) \ge
 \big\{ \bm{\sigma}^{(1-\lambda)}_{K/(N-1)}(\|y-x\|) \rho(x)^{1/(N-1)}
 +\bm{\sigma}^{(\lambda)}_{K/(N-1)}(\|y-x\|) \rho(y)^{1/(N-1)} \big\}^{N-1}
\end{equation}
for all $x,y \in \R^n$ and $\lambda \in (0,1)$ if $N \in (-\infty,0] \cup [n,\infty);$
\begin{equation}\label{eq:normCD'}
\log\rho\big( (1-\lambda)x+\lambda y \big) \ge
 (1-\lambda) \log\rho(x) +\lambda \log\rho(y)
 +\frac{K}{2}(1-\lambda)\lambda \|y-x\|^2
\end{equation}
if $N=\infty$.
Then we have
\[ \cI_{(\R^n,\|\cdot\|,\fm)}(\theta)
 \ge \Lambda_{(\R^n,\|\cdot\|)}^{-1} \cdot \cI_{K,N,\infty}(\theta) \]
for all $\theta \in [0,1]$,
where we consider the distance function $d(x,y):=\|y-x\|$ associated with $\|\cdot\|$
and $\Lambda_{(\R^n,\|\cdot\|)}:=\sup_{x \neq 0}\|{-}x\|/\|x\|$.
\end{corollary}

\begin{proof}
We first see that,
for any $\ve>0$, there exists a Minkowski norm $\|\cdot\|_{\ve}$
$($in the sense of Definition~$\ref{df:Fstr})$ with
$\|x\| \le \|x\|_{\ve} \le (1+\ve) \|x\|$ for all $x \in \R^n$.
To this end, we modify $\|\cdot\|$ in two steps.
Denote by $|\cdot|$ the Euclidean norm of $\R^n$.
We first employ a rotationally symmetric mollifier $\phi \in \cC^{\infty}(\R^n)$
such that $\supp\phi$ is included in the unit ball with respect to $|\cdot|$
and $\int_{\R^n} \phi \,d\fL^n=1$.
For $\delta>0$ and $y \in \R^n$, we define
$\phi_{\delta}(y):=\phi(\delta^{-1} \cdot y)/\delta^n$ and
\[ \|x\|'_{\delta} :=\int_{\R^n} \|x-y\| \phi_{\delta|x|}(y) \,\fL^n(dy) \]
for $x \in \R^n \setminus \{0\}$, and $\|0\|'_{\delta}:=0$.
Then $\|\cdot\|'_{\delta}$ is $\cC^{\infty}$ on $\R^n \setminus \{0\}$ and
\[ \|\cdot\| \le \|\cdot\|'_{\delta} \le \|\cdot\|'_{\delta'} \qquad
 \text{for}\ 0<\delta<\delta'. \]
Using this monotonicity we also observe the convexity of $\|\cdot\|'_{\delta}$.
Next we consider
\[ \|x\|''_{\delta} :=\sqrt{(\|x\|'_{\delta})^2 +\delta|x|^2}, \qquad x \in \R^n. \]
Clearly $\|\cdot\|''_{\delta}$ is strongly convex and enjoys $\|\cdot\|'_{\delta} \le \|\cdot\|''_{\delta}$.
Choosing small $\delta>0$ and letting $\|\cdot\|_{\ve}:=\|\cdot\|''_{\delta}$ shows the claim.

The space $(\R^n,\|\cdot\|_{\ve},\fm)$ satisfies
\eqref{eq:normCD} or \eqref{eq:normCD'} by replacing $K$ with $\min\{K,(1+\ve)^{-2}K\}$.
Therefore the assertion follows from Theorem~\ref{th:isop} by letting $\ve \downarrow 0$.
$\qedd$
\end{proof}

We remark that
the above proof heavily relies on the special properties of normed spaces;
The bi-Lipschitz deformations $\|\cdot\|_{\ve}$ still have straight lines as geodesics,
and they do not destroy the curvature bounds \eqref{eq:normCD}, \eqref{eq:normCD'}.
Hence it is unclear if one can generalize Theorem~\ref{th:isop} to
non-smooth and non-strongly convex Finsler structures.

\section{Further problems}\label{sc:prob}%%%%%%%%%%%%%%%%%%%%%%
%%%%%%%%%%%%%%%%

This final section is devoted to further problems on isoperimetric inequalities on Finsler manifolds.
We refer to \cite[\S 6]{Kl} for other kinds of open problems related to needle decompositions.
(The first problem suggested in \cite[\S 6]{Kl} is actually a generalization to Finsler manifolds,
which is resolved by \cite{CM} and the present article.)

\begin{enumerate}[(A)]
\item
\emph{$N \in (0,1)$ and $N=1$ cases.}
Because of recent works \cite{Mineg,Kl,Wy3} on weighted Riemannian manifolds
with $\Ric_N \ge K$ for $N \in (0,1)$ and even $N=1$,
it is natural to try to generalize the entropy-based curvature-dimension condition $\CD(K,N)$
to $N \in (0,1)$ and $N=1$.
It seems necessary to modify the entropy $S_N$,
since the generating function $h(s)=s^{(N-1)/N}$ is still convex
but $h(0) \neq 0$ for $N \in (0,1)$.

\item
\emph{Concavity of isoperimetric profile.}
In the weighted Riemannian case,
we have a more precise control of the isoperimetric profile.
Under $\Ric_N \ge K$, the differential inequality
\begin{equation}\label{eq:Icon}
\left( \cI_{(M,g,\fm)}^{N/(N-1)} \right)'' \le -\frac{KN}{N-1} \cI_{(M,g,\fm)}^{(2-N)/(N-1)}
\end{equation}
holds in the weak sense (see \cite{Bay,BR,Mo}).
Then, together with the asymptotic analysis at $\theta \to 0$,
we have the sharp isoperimetric inequality.
The standard technique showing this inequality uses the regularity theorem
of isoperimetric minimizers, that is not yet established in the Finsler situation.
If it is generalized, then one could follow the lines of the Riemannian case
(even in the non-reversible case).
Alternatively, it is very interesting if one can derive \eqref{eq:Icon}
by using needle decompositions.

\item
\emph{Asymmetric $\CD$-spaces.}
As a general framework including both $\CD$-spaces in \cite{CM} and
non-reversible Finsler manifolds in the present article,
it would be worthwhile to consider asymmetric metric measure spaces
satisfying the curvature-dimension condition.
Many results could be generalized verbatim,
while we need a careful treatment on the relation between $d(x,y)$ and $d(y,x)$,
they are always locally comparable in the Finsler case by virtue of the smoothness of Finsler structures.

\item
\emph{Other curvature bounds.}
Our construction of needle decompositions used only the manifold structure of the space
and the curvature bound $\Ric_N \ge K$ came into play only from \S \ref{ssc:quan}.
Thus there would be applications of this technique to other curvature bounds,
for instance, Hadamard manifolds
(simply connected Riemannian manifolds of nonpositive sectional curvature)
or its Finsler counterparts.
Then, since needles are weighted spaces (even if the original manifold is unweighted),
we need to develop the theory of sectional curvature for weighted spaces.
This is not as well understood as the weighted Ricci curvature
(we refer to \cite{Wy1,KW,Wy2} for recent attempts in such a direction).

\item
\emph{Rigidity.}
In \cite{CM}, some rigidity results for the isoperimetric inequality
in positively curved spaces were obtained with the help of
the maximal diameter rigidity in \cite{Ke} for metric measure spaces satisfying
the Riemannian curvature-dimension condition.
In the Finsler setting, rigidity is a more challenging problem
and we know only a few including the following.
\begin{enumerate}[(a)]
\item
The maximal diameter under $\Ric_N \ge K>0$, $N \in [n,\infty)$,
implies the spherical suspension structure.
This is essentially included in the framework of the measure contraction property (see \cite{Omcp2}).

\item
Under $\Ric_N \ge 0$ with $N \in [n,\infty]$, the existence of a straight line implies that
the space splits off the real line $\R$ (see \cite{Osplit}).
This is a generalization of Cheeger--Gromoll's classical \emph{splitting theorem}.
\end{enumerate}

Now, it would be interesting to consider rigidity problems, for instance,
in the setting of Corollary~\ref{cr:LG} for reversible spaces.

\item
\emph{Semigroup approach to Bakry--Ledoux's inequality.}
Bakry--Ledoux's isoperimetric inequality \cite{BL}
was shown in the more abstract framework of linear semigroups
satisfying an inequality corresponding to the dimension-free Bochner inequality
(this is the original ``curvature-dimension condition'' $\CD(K,\infty)$).
Applying this result to the weighted Laplacian
on weighted Riemannian manifolds gives the Riemannian case of Corollary~\ref{cr:BL}.
In the Finsler case, although the natural Laplacian is not linear (see \cite{OShf}),
a kind of Bochner--Weitzenb\"ock formula was established in \cite{OSbw}
and has a number of applications including the aforementioned splitting theorem in \cite{Osplit}.
It seems worthwhile to consider whether this alternative approach works or not,
that may improve the estimate in Corollary~\ref{cr:BL}.
\end{enumerate}
\medskip

{\it Note added in proof.}
In our subsequent paper \cite{Oisop},
the sharp isoperimetric inequality under $\Ric_{\infty} \ge K>0$
is resolved along the lines of Bakry--Ledoux's $\Gamma$-calculus approach,
exactly suggested in (F) above.
A similar technique further yields several functional inequalities
on non-reversible Finsler manifolds (\cite{Ofunc}, see also the survey \cite{Onlga}),
as well as Bakry--Ledoux's isoperimetric inequality on $\RCD(K,\infty)$-spaces
with $K>0$ (\cite{AM}, note that in \cite{CM} they assumed $N \in (1,\infty)$).

\renewcommand{\refname}{{\large References}}
{\small%%%

}

\end{document}